\documentclass[reqno, 11pt]{amsart}
\usepackage{amsmath,amssymb,amsfonts,amsbsy}
\usepackage{graphicx}
\usepackage{psfrag}
\usepackage{color}
\usepackage{ulem}

\usepackage[mathscr]{euscript}
\usepackage{calrsfs}
\newtheorem{thm}{Theorem}[section]
 \newtheorem{cor}[thm]{Corollary}
 \newtheorem{lem}[thm]{Lemma}

 \newtheorem{defn}[thm]{Definition}
 \newtheorem{rem}[thm]{Remark}

\newcommand{\mysection}[1]{\section{#1} \setcounter{equation}{0}}
\newcommand{\be}{\begin{equation} \label}
\newcommand{\ee}{\end{equation}}
\newcommand{\bea}{\begin{eqnarray}\label}
\newcommand{\eea}{\end{eqnarray}}
\newcommand{\bas}{\begin{eqnarray*}}
\newcommand{\eas}{\end{eqnarray*}}
\newcommand{\bit}{\begin{itemize}}
\newcommand{\eit}{\end{itemize}}
\newcommand{\nn}{\nonumber}
\newcommand{\R}{\mathbb{R}}
\newcommand{\N}{\mathbb{N}}
\newcommand{\Z}{\mathbb{Z}}
\newcommand{\pO}{\partial\Omega}
\newcommand{\dN}{\partial_\nu}
\newcommand{\eps}{\varepsilon}
\newcommand{\dist}{{\rm dist}}
\newcommand{\supp}{{\rm supp} \, }
\newcommand{\subsubset}{\subset\subset}
\newcommand{\wto}{\rightharpoonup}
\newcommand{\io}{\int_{\Omega}\!}
\newcommand{\iop}{\int_{\Om'}\!}
\def\bge{\begin{eqnarray}}
\def\ege{\end{eqnarray}}
\def\bgee{\begin{eqnarray*}}
\def\egee{\end{eqnarray*}}

\newcommand{\abs}{\\[5pt]}
\newcommand{\Abs}{\\[5mm]}

\newcommand{\ue}{u_{\eps}}
\newcommand{\uen}{u_{0\eps}}
\newcommand{\uet}{u_{\eps t}}
\newcommand{\uu}{\underline{u}}
\newcommand{\ou}{\overline{u}}
\newcommand{\rhoeps}{\rho_\eps\!}
\newcommand{\na}{\nabla}
\newcommand{\Om}{\Omega}
\newcommand{\Ombar}{\overline{\Om}}
\newcommand{\intnT}{\int_0^T}
\newcommand{\dOm}{{\pO}}
\newcommand{\bdry}{|_{\dOm}}
\newcommand{\intdom}{\int_{\dOm}}
\newcommand{\norm}[2][]{\left\|#2\right\|_{#1}}
\newcommand{\Liom}{L^\infty(\Om)}
\newcommand{\ddt}{\frac{d}{dt}}
\newcommand{\del}{\partial}
\newcommand{\Ttilde}{\widetilde{T}}
\newcommand{\uhat}{\widehat{u}}
\newcommand{\set}[1]{\left\{#1\right\}}
\newcommand{\smallset}[1]{\{#1\}}
\newcommand{\phii}{\varphi}
\newcommand{\Lap}{\Delta}
\newcommand{\Tmax}{T_{max}}

\newcommand{\dom}{\del\Om}
\newcommand{\thetaa}{\vartheta}
\newcommand{\sub}{\subset}
\newcommand{\al}{\alpha}

\newcommand{\ba}{\begin{array}}
\newcommand{\ea}{\end{array}}

\def\bge{\begin{eqnarray}}
\def\bgee{\begin{eqnarray*}}
\def\ege{\end{eqnarray}}
\def\egee{\end{eqnarray*}}
\DeclareMathOperator*{\esssup}{ess\,sup}
\allowdisplaybreaks

\begin{document}

\title[A non-local degenerate parabolic problem]
{On a degenerate non-local parabolic problem describing infinite dimensional replicator dynamics}
\author{Nikos I. Kavallaris}
\address{
 Department of Mathematics, University of Chester, Thornton Science Park
Pool Lane, Ince, Chester  CH2 4NU, UK
}

\email{n.kavallaris@chester.ac.uk}
\author{Johannes Lankeit}
\address{Institut f\"ur Mathematik, Universit\"at Paderborn,
Warburger Str. 100, 33098 Paderborn, Germany}
\email{jlankeit@math.uni-paderborn.de}
%
%\and
%
\author{Michael Winkler}
\address{Institut f\"ur Mathematik, Universit\"at Paderborn,
Warburger Str. 100, 33098 Paderborn, Germany}
\email{michael.winkler@math.uni-paderborn.de}
\date{\today}
\begin{abstract}
\noindent
We establish the existence of locally positive weak solutions to the homogeneous Dirichlet problem for
\[
 u_t = u \Delta u + u \int_\Omega |\nabla u|^2
\]
in bounded domains $\Om\sub\R^n$ and prove that solutions converge to $0$ if the initial mass is small, whereas they undergo blow-up in finite time if the initial mass is large. We show that in this case the blow-up set coincides with $\overline{\Omega}$, i.e. the finite-time blow-up is global.\abs
  {\bf Math Subject Classification (2010):} 35K55, 35K65, 35B44,  91A22.\\
% 35K55 Nonlinear parabolic equations
% 35K65 Degenerate parabolic equations
% 35B44 Blow-up
% 35B40 Asymptotic behaviour of solutions
% 91A22 Evolutionary games
\end{abstract}
\maketitle
\mysection{Introduction}
In a bounded domain $\Omega \subset \R^N,\;N \ge 1$, we consider nonnegative solutions to the
quasilinear degenerate and nonlocal parabolic problem
\be{0}
	\left\{ \begin{array}{ll}
	\displaystyle
	u_t = u\Delta u + u \int_\Omega |\nabla u|^2, \qquad & x\in\Omega, \ t>0, \\[2mm]
	u(x,t)=0, \qquad & x\in\pO, \ t>0, \\[2mm]
	u(x,0)=u_0(x), \qquad & x\in\Omega,
	\end{array} \right.
\ee
which arises in a game theoretical description of replicator dynamics in the case of a  Bomze-type
infinite dimensional setting \cite{B90} by pursuing a  modeling procedure introduced in \cite{KPY08, KPXY10, PS09} and which actually
assumes steep payoff-kernels of Gaussian type.
For completeness in this direction we include a concise derivation of the particular parabolic
equation in (\ref{0}) in the Appendix A.\abs
{\bf Strongly degenerate diffusion meets nonlocal gradient sources.} \quad
From a mathematical perspective, the evolution in (\ref{0}) is governed by two characteristic mechanisms, each of which
already gives rise to considerable challenges on its own.
Firstly, diffusion in (\ref{0}) is strongly degenerate at small densities in the sense that near points where $u=0$ 
%hence particularly near the spatial boundary,
typical diffusive effects are substantially inhibited.
Indeed, already in the unforced counterpart of (\ref{0}) with general power-type degeneracy, as given by
\be{deg_p}
	u_t=u^p\Delta u
\ee
with $p>0$, it is known that the particular value $p=1$, corresponding to the choice in (\ref{0}), marks a borderline
between somewhat mild degeneracies and strongly degenerate diffusion: In the case when $p<1$, namely, (\ref{deg_p})
allows for a transformation into the porous medium equation $v_t=\Delta v^m$ with $m:=\frac{1}{1-p}>1$,
thus meaning that in this case unique global continuous
weak solutions to the associated Dirichlet problem exist for all reasonably
regular nonnegative initial data (\cite{aronson}), and that these eventually become positive and smooth, and hence classical,
inside $\Omega$ (\cite{bertsch_peletier}).
%; after an appropriate waiting time, all these solutions even enter the
%cone
%${\cal K}:=\{\varphi: \Omega \to \R \ | \ \varphi(x)\ge c \dist(x,\pO) \mbox{ for all } x\in\Omega \mbox{ and some } c>0\}$
%(\cite{bertsch_pos}),
%thus reflecting a diffusion-driven effect generalizing the Hopf boundary point property in non-degenerate diffusion
%processes.
If $p\ge 1$ then nonnegative global weak solutions can still be constructed for any nonnegative continuous initial data,
but they need no longer be continuous (\cite{BDalPU}) nor uniquely determined by the initial data
(\cite{LDalP}),
and moreover their spatial support	%, the spatial support of such weak solutions
will not increase with time (\cite{BU, LDalP, win_trieste}).\\
Even in the case when one resorts to continuous initial data which are strictly positive throughout $\Omega$,
in which in fact unique classical solutions exist for any $p>0$, the value $p=1$ corresponds to a critical
strength of degeneracy in that for $p<1$, after an appropriate waiting time, all solutions will enter the
cone
${\mathcal K}:=\{\varphi: \Omega \to \R \ | \ \varphi(x)\ge c \dist(x,\pO) \mbox{ for all } x\in\Omega \mbox{ and some } c>0\}$
(\cite{bertsch_peletier}),
thus reflecting a diffusion-driven effect generalizing the Hopf boundary point property in non-degenerate diffusion
processes.
In the case $p\ge 1$, however, solutions to (\ref{deg_p})
emanating from initial data which are suitably small near $\pO$ will
never enter ${\mathcal K}$ (\cite{win_AMUC}).\abs
Now in (\ref{0}), this degenerate diffusion process interacts with a spatially
nonlocal source which is such that unlike in
large bodies of the literature on related nonlocal parabolic equations (\cite{quittner_souplet}),
even basic questions concerning local solvability appear to be far from trivial:
Indeed, in light of an expected loss of appropriate solution regularity due to strongly degenerate diffusion,
even for smooth initial data
it seems a priori unclear whether solutions can be constructed which allow for a meaningful definition
of the Dirichlet integral $\io |\nabla u|^2$ for positive times.
This is in stark contrast to most nonlocal parabolic problems previously studied, in which either
diffusion is non-degenerate and hence such first-order expressions are controllable by $L^\infty$ bounds for solutions
at least for small times, such as e.g.~in the semilinear problem
\bas
	u_t = \Delta u + u^m\left(\int_\Omega |\nabla u|^2\,dx\right)^r	%, \qquad x\in\Omega, \ t>0, \\[2mm]
\eas
studied for $m\geq 1$, $r>0$ in (\cite{Sou02}),
or the nonlocal terms involve only zero-order expressions which thus in a natural manner
also in cases of degeneracies as in (\ref{deg_p})
allow for local theories based on extensibility criteria in $L^\infty(\Omega)$ only (see \cite{deng_duan_xie, souplet_BUnonloc} and also the book
\cite{quittner_souplet}).\Abs
{\bf Main results.} \quad
Previous mathematical studies on the PDE in (\ref{0}) have concentrated on analyzing self-similar solutions only.
In \cite{KPY08}, the authors constructed self-similar solutions in the case $\Omega=\R$,
and in \cite{PS09} the same could be achieved in the multi-dimensional case $\Omega=\R^N$ with $N\ge 2$.
More recently, the authors in \cite{PV14} investigated the existence of self-similar solutions in the one-dimensional case
in a closely related problem in which the Laplacian is perturbed by a time-dependent term containing the first derivative
as well; all these self-similar solutions are shown to be singular and to approach Dirac-type
distributions as $t\searrow 0$.\abs
The goals of the present work consist in developing a fundamental theory of local solvability for (\ref{0}),
and in providing a first step toward an understanding of the qualitative solution behaviour.
In order to formulate our results, let us concretize the specific setting within which (\ref{0}) will be studied
by requiring that throughout the sequel, $\Omega$ denotes a bounded domain in $\R^N, N\ge 1$, with smooth boundary,
and by introducing the solution concept that we shall pursue as follows.
\begin{defn}\label{defi44}
  Let $T\in (0,\infty]$. By a {\em weak solution} of (\ref{0}) in $\Omega \times (0,T)$ we mean a nonnegative function
  \bas
	u\in L^\infty_{loc}(\bar\Omega \times [0,T)) \, \cap \, L^2_{loc}([0,T); W_0^{1,2}(\Omega))
	\qquad \mbox{with} \qquad u_t \in L^2_{loc}(\bar\Omega \times [0,T)),
  \eas
  which satisfies
  \be{0vw}
	-\int_0^T\!\!\!\! \io\!\! u\varphi_t\,dxdt + \int_0^T\!\!\!\! \io\!\! \nabla u \cdot \nabla (u\varphi)\,dxdt
	=\io\!\! u_0 \varphi(\cdot,0)\,dx + \int_0^T\!\!\! \Big( \io\!\! u\varphi\,dx \Big) \cdot \Big( \io\!\! |\nabla u|^2 \,dx\Big)dt
  \ee
  for all $\varphi \in C_0^\infty(\Omega \times [0,T))$.\\
  A weak solution $u$ of (\ref{0}) in $\Omega \times (0,T)$ will be called {\em locally positive} if
  $\frac{1}{u} \in L^\infty_{loc}(\Omega \times [0,T])$.
\end{defn}
\begin{rem} \quad
  Since $u \in L^2_{loc}([0,T);W_0^{1,2}(\Omega))$ and $u_t \in L^2_{loc}(\bar\Omega \times [0,T))$
  imply that $u \in C^0([0,T);L^2(\Omega))$, (\ref{0vw}) is equivalent
  to requiring that $u(\cdot,0)=u_0$, and that
  \be{0w}
	\int_0^T \io u_t \varphi \,dx\,dt+ \int_0^T \io \nabla u \cdot \nabla (u\varphi)\,dx\,dt
	=\int_0^T \Big( \io u\varphi\,dx \Big) \cdot \Big( \io |\nabla u|^2\,dx \Big)\,dt
  \ee
  holds for any $\varphi \in C_0^\infty(\Omega \times (0,T))$.
\end{rem}
In order to construct such locally positive weak solutions, we shall assume that the initial data satisfy
\bit
\item[(H1)]
  $u_0 \in L^\infty(\Omega) \cap W_0^{1,2}(\Omega)$ and
\item[(H2)]
  $u_0 \ge 0$ and $\frac{1}{u_0} \in L^\infty_{loc}(\Omega)$ as well as
\item[(H3)]
 there exists $L>0$ such that $\norm[\Phi,\infty]{u_0}\leq L$.
\eit
Here and below, for a measurable function $v\colon \Om\to\R$ we have set
\[
 \norm[\Phi,\infty]{v}:= \esssup_{x\in \Omega} \left|\frac{v}{\Phi}\right|,
\]
where $\Phi\in C^2(\Ombar)$ denotes the solution to
\begin{equation}\label{eq:defPhi}
 -\Delta \Phi = 1\quad \mbox{ in }\Om, \qquad \Phi\bdry=0.
\end{equation}
Note that according to the Hopf boundary point lemma, requiring $\norm[\Phi,\infty]{u_0}$ to be finite is an equivalent
way to ask for the possibility of estimating $u_0$ by a multiple of the function measuring the distance of
a point to $\dOm$.\abs
In this framework, the first of our main results indeed asserts local existence of locally positive weak solutions,
along with a favorable extensibility criterion only involving the norm of the solution in $L^\infty(\Omega)$.
\begin{thm}\label{thm:Tmax}
  Let $u_0$ satisfy (H1)-(H3).
  Then there exist $\Tmax\in(0,\infty]$ and a locally positive weak solution $u$ to \eqref{0} in $\Om\times(0,\Tmax)$
  which satisfies
  \be{ext_crit}
 	\mbox{either } \Tmax=\infty \quad \mbox{ or } \quad \limsup_{t\nearrow\Tmax} \norm[\Liom]{u(\cdot,t)}=\infty,
  \ee
  and which is such that for each smoothly bounded subdomain $\Om'\subsubset\Om$ there exists $C_{\Om'}>0$ with
  \bea{leq}
	\io |\na u(\cdot,t)|^2
	\le
	\io|\na u_0|^2 \cdot \exp\Bigg[\frac{1}{2C_{\Om'}}\left(\sup_{\tau\in(0,t)}\io u(\cdot,\tau)\right)
	\left(\iop\phi\ln u(\cdot,t)-\iop \phi\ln u_0+\int_0^t\iop  u \right) \Bigg]
  \eea
  as well as
  \be{leqnormphi}
   	\norm[\Phi,\infty]{u(\cdot,t)} \leq \max \set{\norm[\Phi,\infty]{u_0},\sup_{\tau\in(0,t)} \io|\na u(x,\tau)|^2\,dx}.
  \ee
  for a.e. $t\in(0,\Tmax)$.
\end{thm}
We emphasize that the extensibility criterion (\ref{ext_crit}) particularly excludes any
{\it gradient blow-up} phenomenon in the sense of finite-time blow-up of $\nabla u$
despite boundedness of $u$ itself.
Indeed, the occurrence of unbounded gradients of bounded solutions appears to be a characteristic qualitative implication
of various types of interplay between diffusion, possibly degenerate, and gradient-dependent nonlinearities
(\cite{angenent_fila96, arrieta_rodriguezbernal_souplet04, li_souplet10, stinner_win08}).\abs
A natural next topic appears to consist in deriving conditions on the initial data which ensure that
the solutions found above either exist for all times, or blow up in finite time.
Here in view of the essentially cubic character of the production term in (\ref{0}) it is not surprising that
this may dominate the smoothing effect of the merely quadratic-type diffusion term
when the initial data are suitably large in an adequate sense;
precedent works indicate that indeed such intuitive considerations are appropriate in related non-degenerate
and degenerate
parabolic equations with local reaction terms (\cite{quittner_souplet,SGKM, stinner_win07, win_critexp} ).\\
As a remarkable feature of the precise structure of this interplay in (\ref{0}), we shall see that actually
a complete classification of all initial data in this respect is possible,
exclusively involving the size of the total initial mass $m:=\io u_0$ as the decisive quantity:
In fact, the second of our main results identifies the value $m=1$ to be critical with regard to
global solvability, and moreover gives some basic information on the asymptotic behaviour of solutions.
\begin{thm}\label{thm:longterm}
  Let $u_0$ satisfy (H1)-(H3), and let $u$ and $\Tmax$ denote the corresponding locally positive weak solution of \eqref{0},
  as well as its maximal time of existence,
  provided by Theorem \ref{thm:Tmax}.\\
  (i) \ If $\io u_0<1$, then $\Tmax=\infty$ and
  \[
 	\io u(x,t)\, dx\to 0
	\qquad \mbox{as } t\to\infty.
  \]
  (ii) \ Suppose that $\io u_0=1$. Then $\Tmax=\infty$ and
  \[
 	\io u(x,t)\,dx=1
	\qquad \mbox{for all } t>0.
  \]
  (iii) \ In the case $\io u_0\, dx >1$, we have $\Tmax<\infty$ and
  \[
 	\limsup_{t\nearrow \Tmax} \io u(x,t)\,dx = \infty.
  \]
\end{thm}
\begin{rem}
The statement $(ii)$ of Theorem \ref{thm:longterm} says that if the initial data $u_0$ is a probability measure then we have conservation of probability in time. This is actually a desired feature of the replicator dynamics model described by \eqref{0}, since $u(\cdot,t)$ stands for a probability distribution of the state of some population of players, see also Appendix A.
\end{rem}

In the situation of Theorem \ref{thm:longterm} (iii) when finite-time blow-up occurs, understanding
the solution behaviour near the respective blow-up time necessarily requires to describe the set of all points
where the solution becomes unbounded.
Accordingly, we shall next be concerned with the blow-up set
\bas
	\mathcal{B} = \Big\{x\in\Ombar &\Big|&
	\mbox{there exists a sequence $(x_k,t_k)_{k\in\N} \subset \Omega\times (0,\Tmax)$ such that } \\[0mm]
	& & x_k\to x, t_k\to\Tmax \mbox{ and } u(x_k,t_k)\to \infty
	\mbox{ as } k\to\infty \Big\}
\eas
of exploding solutions.
In numerous related equations, involving either linear or degenerate diffusion,
blow-up driven by local superlinear production terms is known to occur in thin spatial sets only, in radial settings typically
reducing to single points (\cite{friedman_mcleod, giga_kohn, SGKM}),
with only few exceptional situations detected
in the literature which lead to {\it regional} or even {\it global} blow-up, thus referring to
cases in which $|{\mathcal B}|>0$ or even ${\mathcal B}=\Ombar$
(cf.~\cite{friedman_mcleod_degenerate, galaktionov_vazquez, lacey, stinner_win07, win_BU_JDE03}, for instance).
In cases of sources which at least partially consist of nonlocal terms, blow-up in sets of positive measure may occur if the relative size of a
possibly contained local contribution at large densities is predominant, as compared to the strength of the respective diffusion term  (\cite{du_xiang, liang_li, liu_li_gao, souplet_nonlocnonlinsource, souplet_unifbuprofile,wangwang}).\abs
Our main result in this direction will reveal that
any of our non-global solutions in fact blow up globally in space, thus indicating a certain balance in the
competition of diffusion and nonlocal production in (\ref{0}):
\begin{thm}\label{thm:blowupset}
  Suppose that $\io u_0\, dx >1$, and let $u$ denote the locally positive weak solution of \eqref{0}
  from Theorem \ref{thm:Tmax}.
  Then $u$ blows up globally in the sense that its blow-up set satisfies $\mathcal{B}=\Ombar$.
\end{thm}
The outline of the paper is as follows. In Section \ref{s2} we introduce an approximate sequence of non-degenerate problems
and derive some estimates for their solutions $\ue$. Here one key step toward the existence proof will consist in deriving the associated approximate variant of \eqref{theestimate} (Lemma \ref{thm:control_gradient_l2}), wich will rely on an energy type argument combined with an analysis of the functional $\iop \phi \ln \ue(\cdot,t)$ for $\Om'\subsubset \Om$, $t>0$ and appropriate $\phi$. Another important observation, based on an integral estimate involving certain singular weights (cf. Lemma \ref{uepsbdry} and in particular \eqref{ps1}), will reveal that the functions $\na \ue$ enjoy a favorable strong compactness property with respect to spatio-temporal $L^2$-norms (cf. \eqref{28.99}), rather than merely the respective weak precompactness feature obtained from corresponding boundedness results. In Section \ref{s3} we study an ODE problem associated with the evolution of the total mass of the solution,
and in dependence on whether this total mass initially is equal, less or greater than $1$, we prove global existence and conservation of the total mass, convergence to zero total mass and finite-time blow-up, respectively.
Finally, in Section \ref{s4} we concentrate on the latter case and examine the corresponding blow-up set of the solution,
and we actually prove that any such blow-up occurs globally in space.
\mysection{Weak solutions: existence and approximation}\label{s2}
Following an approach well-established in the context of degenerate parabolic equations, we aim at constructing a solution $u$ to (\ref{0}) as the limit of solutions to certain regularized problems. For this purpose, let us fix a sequence $(\eps_j)_{j\in\N} \subset (0,1)$
such that $\eps_j \searrow 0$ as $j\to\infty$, and a sequence $(u_{0\eps})_{\eps=\eps_j} \subset C^3(\bar\Omega)$ with
the properties
\be{a1}
	u_{0\eps} \ge \eps \ \mbox{in } \Omega, \qquad u_{0\eps}=\eps \ \mbox{on } \pO, \qquad
	\Delta u_{0\eps}=-\io |\nabla u_{0\eps}|^2 \ \mbox{on } \pO\quad \mbox{ for all } \eps\in(\eps_j)_{j\in\N}
\ee
and
\be{ae}
      \limsup_{\eps=\eps_j\searrow 0} \norm[\Phi,\infty]{u_{0\eps}-\eps}\leq L,
\ee
%with $L$ as in (H3)
with $L>\max\set{\io|\na u_0|^2, \norm[\Phi,\infty]{u_0}}$, cf. (H3),
as well as
\be{a3}
	\mbox{for any compact set } K\subset \Om \mbox{ there is }C_K>0 \mbox{ such that } \liminf_{\eps\searrow 0} \inf_K \uen \geq C_K, %u_0 \le u_{0\eps} \le u_0+1 \quad \mbox{a.e.~in } \Omega
\ee
and such that moreover
\be{a5}
 \uen\to \ue \quad \mbox{in } W^{1,2}(\Om) \qquad \mbox{as }\eps=\eps_j\searrow 0
\ee
and
\be{a6}
 \int \uen = \int u_0 \qquad \mbox{ for all }\eps\in(\eps_j)_{j\in\N}.
\ee
A necessary first observation is that such an approximation actually is possible.
\begin{lem}\label{lem:existenceofapproximation}
 Let $u_0$ satisfy (H1)-(H3).
 Then there is a sequence $(u_{0\eps})_{\eps\in(\eps_j)_{j\in\N}} \subset C^3(\bar\Omega)$ having the properties \eqref{a1}-\eqref{a6}.
\end{lem}
\begin{proof}
By modification of the usual mollification procedure (cf. \cite[Section I §3]{Wloka}) commonly employed to obtain \eqref{a5} it is possible to obtain the other properties as well. More precisely, we set
\[
 \uen = \eps + C (1-\rho)\Phi+\rho(\phii+\alpha\thetaa),
\]
where $\phii\in C_0^\infty(\Om)$ is a mollified version of $u_0$ (after ``locally shifting $u_0$ towards the interior of the domain''), $\rho\in C_0^\infty(\Om)$, $0\leq\rho\leq1$, such that the supports of $\na \rho$ and $\phii$ are disjoint, $0\leq \thetaa\in C_0^\infty$ with $\io \thetaa=1$ (in order to adjust \eqref{a6}), $\Phi$ is the solution to $-\Delta\Phi=1$ in $\Om$, $\Phi=0$ on $\dOm$ (for achieving the third property in \eqref{a1}), and $C$ and $\alpha$ are appropriately adjusted constants, depending on $\eps$ as well as several different integrals containing the functions $\Phi$, $\rho$, $\thetaa$, their gradients, and $u_0$.\\
For a slightly more detailed version of the proof, we refer the reader to the appendix.
\end{proof}

For $\eps\in(\eps_j)_{j\in\N}$, we consider the regularized problem
\be{0eps}
	\left\{ \begin{array}{ll}
	u_{\eps t}=\ue \Delta \ue +  \ue\cdot\rhoeps\big(\io |\nabla\ue|^2\big), \qquad &x\in\Omega, \ t>0, \\[2mm]
	\ue(x,t)=\eps, \qquad &x\in\pO, \ t>0, \\[2mm]
	\ue(x,0)=u_{0\eps}(x), \qquad &x\in\Omega,
	\end{array} \right.
\ee
where
\bas
	\rhoeps(z):=\min \Big\{ z,\ \frac{1}{\eps} \Big\} \qquad \mbox{for } z\ge 0.
\eas
\begin{lem}\label{lem58}
  For all sufficiently small $\eps \in (\eps_j)_{j\in\N}$, problem (\ref{0eps}) has a unique classical global-in-time solution
  $\ue\in C^{2,1}(\Ombar\times[0,\infty))$.
\end{lem}
\begin{proof}
  To prove the uniqueness statement for all $\eps$, we assume that both $u_1$ and $u_2$ are classical solutions of (\ref{0eps}) from the indicated class
  in $\Omega \times (0,T)$ for some $T>0$. Then $w:=u_1-u_2$ satisfies $w=0$ on $\pO$ and at $t=0$, and
   \be{eq:2.13b}
 	w_t=u_1 \Delta w + \Delta u_2 \cdot w + \rhoeps \Big( \io |\nabla u_2|^2\Big) \cdot w
 	+u_1 \cdot \Big[ \rhoeps\Big( \io |\nabla u_1|^2 \Big) - \rhoeps\Big( \io |\nabla u_2|^2\Big) \Big]
   \ee
  for $t\in (0,T)$.
  Now given $T' \in (0,T)$, we can find a constant $M>0$ such that $u_1,|\nabla u_1|, u_2$ and $|\nabla u_2|$ are bounded above by $M$ in $\Omega \times (0,T')$, since $u_1,u_2$ are classical solutions. Thus, by H\"older's inequality and the pointwise estimate $\Big| |\nabla u_1|-|\nabla u_2|\Big| \le
  |\nabla (u_1-u_2)|$, we obtain
  \begin{align}\label{eq:2_10a}
	\bigg| \rhoeps\Big( \io |\nabla u_1|^2 \Big) - \rhoeps \Big(\io |\nabla u_2|^2\Big) \bigg|\nn
	\le& \|\rhoeps'\|_{L^\infty((0,\infty))} \cdot \Big| \io \big(|\nabla u_1|^2 -  |\nabla u_2|^2\big) \Big| \nn\\
	\leq& \io \Big||\nabla u_1| -  |\nabla u_2|\Big|\cdot \big(|\nabla u_1| + |\nabla u_2|\big)\nn\\
	\le& 2M \io |\nabla w| \nn \\
	\le& 2M |\Omega|^\frac{1}{2} \cdot \Big( \io |\nabla w|^2 \Big)^\frac{1}{2}
  \end{align}
  for all $t\in (0,T')$, because $\|\rhoeps'\|_{L^\infty((0,\infty))}\leq 1.$
  Upon multiplying \eqref{eq:2.13b} by $w$ and integrating over $\Omega$ we see that for $t\in(0,T')$
%as well as using Young's inequality and the fact that $u_1 \ge \eps$ by comparison, we see that
  \bea{58}
	\frac{1}{2} \frac{d}{dt} \io w^2
	%&=& \green{\io ww_t\nn}\\
 &=&\io u_1\Delta w w +\io w^2\Delta u_2+\io w^2 \rhoeps\left(\io |\nabla u_2|^2\right) \\&&+ \io wu_1\left[\rhoeps\left(\io |\nabla u_1|^2\right)-\rhoeps\left(\io|\nabla u_2|^2\right)\right]\nn\\
&\leq& - \io u_1|\nabla w|^2 - \io \nabla u_1\nabla w w - 2\io w \nabla w\nabla u_2\nn\\
&& +\io w^2\rhoeps\left(\io |\na u_2|^2\right) + \io |w| u_1\left\lvert \rhoeps\left(\io |\nabla u_1|^2\right)-\rhoeps\left(\io|\nabla u_2|^2\right)\right\rvert\nn.
\eea
Together with Young's inequality, \eqref{eq:2_10a} and the facts that $u_1\geq \eps$ and $\rhoeps(s)\leq \frac1\eps$ for all $s>0$, this entails
\bea{58.0}
\frac{1}{2} \frac{d}{dt} \io w^2&\leq&-\eps\io |\nabla w|^2 + \frac\eps4\io |\nabla w|^2+\frac1\eps\io w^2|\nabla u_1|^2 +\frac{\eps}2\io |\nabla w|^2+\frac 8\eps\io w^2|\nabla u_2|^2\nn\\
&& + \frac1\eps \io w^2+ 2M|\Omega|^{\frac12} \left(\io |\nabla w|^2\right)^\frac12 \io |w|u_1\nn
\eea
for $t\in(0,T')$. The choice of $M$ now ensures that
\begin{align}\label{58.1}
\frac{1}{2} \frac{d}{dt} \io w^2
\leq& -\frac\eps4\io |\nabla w|^2 +\frac{M^2}\eps\io w^2 +\frac {8M^2}\eps\io w^2 + \frac1\eps \io w^2\nn\\
&+ 2M|\Omega|^{\frac12} \left(\io |\nabla w|^2\right)^\frac12 \left(\io |w|^2\io u_1^2\right)^\frac12\nn\\
%&\leq& \green{-\frac\eps4\io |\nabla w|^2 +\frac{M^2}\eps\io w^2 +\frac {8M^2}\eps\io w^2+ \frac1\eps \io w^2+ 2M^2|\Omega| (\io |\nabla w|^2)^\frac12 (\io |w|^2)^\frac12\nn}\\
\leq& -\frac\eps4\io |\nabla w|^2 +\frac {9M^2+1}\eps\io w^2 + \frac\eps4 \io |\nabla w|^2+ \frac{4M^4|\Omega|^2}{\eps} \io |w|^2
%&\leq&  \big(\frac{M^2}\eps +\frac {8M^2}\eps + \frac1\eps + \frac{2^2M^4|\Omega|^2}{\eps} \big) \io w^2
%}
% 	&\le& - \eps \io |\nabla w|^2 + \frac{\eps}{4} \io |\nabla w|^2 + \frac{1}{\eps} \io |\nabla u_1|^2 w^2 \nn\\
% 	& & + \frac{\eps}{2} \io |\nabla w|^2 + \frac{8}{\eps} \io |\nabla u_2|^2 w^2 \nn\\
% 	& & + \color{blue}{\io |\nabla u_1|^2 w^2} + \frac{\eps}{4} \io |\nabla w|^2
% 	+ {\color{blue}\frac{M^2|\Omega|}{2\eps}} \Big(\io u_1^2 \Big) \cdot \Big( \io w^2 \Big)
\end{align}
  for $t\in (0,T')$, so that (\ref{58.1}) finally turns into
  \bas
	\frac{1}{2} \frac{d}{dt} \io w^2
	\le \Big(\frac{9M^2+1}{\eps} + \frac{4M^4 |\Omega|^2}{\eps} \Big) \cdot \io w^2
  \eas
  for all $t\in (0,T')$.

  Integrating this ODI yields that $w\equiv 0$ in $\Omega \times (0,T')$ and hence also in
  $\Omega \times (0,T)$, because $T'<T$ was arbitrary.\abs
  It remains to be shown that for all $T>0$, (\ref{0eps}) is classically solvable in $\Omega \times (0,T)$ provided $\eps$
  is sufficiently small. To this end, fix $T>0$ and
  let $\eps \in (\eps_j)_{j\in\N}$ be so small that $\io|\na \ue|^2<\frac1\eps$, which is possible due to \eqref{a5}.
 By \cite[Thm. V.1.1]{LSU}, there are $K_1>0$ and $\theta>0$ such that any classical solution $w$ to the problem
\[
 w_t=w\Delta w + c(x,t) \mbox{ in } \Om\times[0,T], \qquad w\bdry=\eps,\qquad w(\cdot,0)=\uen
\]
with $c\in L^\infty(\Om\times(0,T))$ fulfilling $0\leq c\leq \frac1\eps\norm[L^\infty(\Om)]{\uen} e^{\frac T\eps}$ which in addition obeys the estimate $\eps\leq w \leq \norm[\infty]{\uen}e^{\frac T\eps}$ satisfies \begin{equation}\label{eq:Holderest1}
\norm[C^{\theta,\frac\theta2}(\Ombar\times {[0,T]})]{w}\leq K_1.
\end{equation}
Fix $\delta>0$. Corresponding to $\theta, K_1$ and $\delta$, there is $K_2$ such that any solution $w$ to
\[
 w_t=a(x,t)\Delta w + b(x,t)\mbox{ in } \Om\times[0,T], \qquad w\bdry=\eps,\qquad w(\cdot,0)=\uen
\]
for some $a\in C^{\theta,\frac\theta2}(\Ombar\times[0,T])$ having the properties $a(x,t)=\eps$ for $(x,t)\in \pO\times[0,T]$, $\eps\leq a\leq \norm[L^\infty]{\uen}e^{\frac T\eps}$, $\norm[C^{\theta,\frac\theta2}(\Ombar\times {[0,T]})]{a} \leq K_1$ and continuous $b$ with $b(x,0)=b_0\in \R$, $\norm[\infty]{b}\leq \frac{K_1}\eps$, by an application of \cite[Thm. 7.4]{Friedman_parabolic} to $w-\uen-tb_0$ fulfils
\begin{equation}\label{eq:Holderest2}
\norm[C^{1+\delta,\frac\delta2}(\Ombar\times {[0,T]})]{w}\leq K_2.
\end{equation}

With this in mind, in the space $X=C^{1+\frac\delta2,\frac\delta4}(\Ombar\times[0,T])$ we consider the set
  \bas
	S:=\Big\{ v\in X \ \Big| \ v \ge \eps \mbox{ in } \Omega \times (0,T), v(\cdot,0)=u_{0\eps} \mbox{ and } \norm[C^{1+\delta,\frac\delta2}(\Ombar\times {[0,T]})]{v}\leq K_2 \Big\},
  \eas
  which is evidently closed, bounded, convex, and compact in $X$. For each $v\in S$, the definition of $\rhoeps$ implies that
  \be{58.11}
	f(t):=\rhoeps \Big( \io |\nabla v(\cdot,t)|^2 \Big), \qquad t\in [0,T],
  \ee
  defines a nonnegative $\frac\delta2$-H\"older continuous function $f$ on $[0,T]$. The choices of $f, S$ and $\eps$ show that $f(0)=\io|\na \uen|^2$ and thus \eqref{a1} ensures that the compatibility condition of first order is satisfied.
  Therefore, the quasilinear, actually non-degenerate parabolic problem
  \be{58.2}
	\left\{ \begin{array}{ll}
	u_{\eps t}=\ue\Delta \ue + f(t)\ue, \qquad &x\in\Omega,\ t>0, \\[2mm]
	\ue(x,t)=\eps, \qquad &x\in\pO, \ t>0, \\[2mm]
	\ue(x,0)=u_{0\eps}(x), \qquad &x\in\Omega,
	\end{array} \right.
  \ee
  possesses a classical solution $\ue\in C^{2,1}(\Ombar\times[0,T])$ by \cite[Thm V.6.1]{LSU}, %page 452
 which, by comparison, satisfies
  \be{58.3}
	\eps \le \ue \le \|u_{0\eps}\|_{L^\infty(\Omega)} \cdot e^{\frac{T}{\eps}} \qquad \mbox{in } \Omega \times (0,T),
  \ee
  because $\uu(x,t):=\eps$ and $\ou(x,t):=\|u_{0\eps}\|_{L^\infty(\Omega)} \cdot e^{\frac{t}{\eps}}$ are easily seen
  to define a sub- and a supersolution of (\ref{58.2}), respectively.\\
 We now introduce a mapping $F\colon S\to X$ by setting $Fv:=\ue,$ where $\ue$ solves \eqref{58.2} with \eqref{58.11}.

 Then defining $c(x,t):=\ue(x,t)f(t)$, $x\in\Om, t\in [0,T]$, this function satisfies $\norm[\infty]{c}\leq \frac1\eps \norm[\infty]{\uen}e^{\frac T\eps}$ and accordingly, as stated in \eqref{eq:Holderest1} above, $\norm[C^{\theta,\frac\theta2}]{Fv}\leq K_1$ for any $v\in S$.\\
 Using $a(x,t):=(Fv)(x,t)$ and $b(x,t):= (Fv)(x,t)\cdot f(t)$, we see that, again, the above considerations are applicable and $\norm[C^{1+\delta,\frac\delta2}(\Ombar\times {[0,T]})]{Fv}\leq K_2$ for any $v\in S$ by \eqref{eq:Holderest2}. In particular, we observe that $FS\subset S$.

Furthermore invoking \cite[IV.5.2]{LSU}, we can conclude the existence of $k>0$ and $K_3>0$ such that

\begin{equation}\label{eq:relcpt}
 \norm[C^{2+\delta,1+\frac\delta2}(\Ombar\times{[0,T]})]{Fv}\leq k\left(\norm[C^{\delta,\frac\delta2}(\Ombar\times {[0,T]})]{Fv\cdot f}+ \norm[C^{2+\delta}(\Ombar\times {[0,T]})]{\uen} + \eps\right) \leq K_3%\;\mbox{ for all }v\in S.
\end{equation}
for all $v\in S$.
To see that $F$ is continuous, we suppose that $(v_k)_{k\in\N} \subset S$ and $v\in S$ are such that
  $v_k \to v$ in $X$. Then $f_k(t):=\rhoeps \big( \io |\nabla v_k(\cdot,t)|^2 \big)$
  satisfies
  \be{58.5}
	f_k \to f \qquad \mbox{in } C^0([0,T])
  \ee
  as $k\to\infty$, with $f$ as given by (\ref{58.11}). By \eqref{eq:relcpt} and the theorem of Arzel\`a-Ascoli, $(F v_k)_{k\in\N}$ is relatively compact in $C^{2,1}(\Ombar\times[0,T])$, and if
  $k_i\to\infty$ is any sequence such that $u_{k_i}:=Fv_{k_i}$ converges in $C^{2,1}(\Ombar\times[0,T])$ to some $w$ as $i\to\infty$, then in
  \bas
	\partial_t u_{k_i} = u_{k_i} \Delta u_{k_i} + f_{k_i}(t) u_{k_i}, \qquad x\in \Omega, \ t\in (0,T),
  \eas
  we may let $k_i\to\infty$ and use (\ref{58.5}) to obtain that $w$ is a classical solution
  of (\ref{58.2}). Since classical solutions of
  (\ref{58.2}) are unique due to the comparison principle, we must have $w=Fv$. We thereby derive that the whole sequence
  $(Fv_k)_{k\in\N}$ converges to $Fv$ and hence conclude that $F$ is continuous. Therefore the Schauder fixed point theorem
  asserts the existence of at least one $\ue \in S$ for which $\ue=F\ue$ holds. Since such a fixed point obviously solves (\ref{0eps}),
  the proof is complete.	
\end{proof}	
The basis of both our existence proof and our boundedness result is formed by the next two lemmata which provide
useful a priori estimates for $\ue$ in terms of certain presupposed bounds.
The first lemma essentially derives a uniform pointwise bound for $\ue$ from a space-time integral estimate for $|\nabla\ue|^2.$
\begin{lem}\label{lem34}
  For all $M>0$ and $B>0$ there exists $C(M,B)>0$ with the following property:
  If
  \be{34.1}
	u_{0\eps} \le M \quad \mbox{in } \Omega \qquad \mbox{and} \qquad \int_0^T \io |\nabla \ue|^2 \le B
  \ee
  holds for some $\eps\in(\eps_j)_{j\in\N}$ and $T \in (0,\infty]$ then we have
  \be{34.2}
	\ue \le C(M,B) \qquad \mbox{in } \Omega \times [0,T).
  \ee
\end{lem}
\begin{proof}
  Our plan is to use a separated function of the form
  \be{34.3}
	\ou(x,t):=y(t) \cdot (M+\Phi(x)), \qquad x\in\bar\Omega, \ t\in [0,T),
  \ee
  as a comparison function, where $M$ is as in the hypothesis of the lemma, $\Phi \in C^2(\bar\Omega)$ is the solution of \eqref{eq:defPhi},
  and $y$ denotes the solution of
  \be{34.5}
	%\left\{ \begin{array}{l}
	y'=-y^2+\big( f(t)+1 \big) \cdot y, \quad t\in (0,T), \qquad %\\[2mm]
	y(0)=1,
	%\end{array} \right.
  \ee
  with $f(t):=\io |\nabla \ue(\cdot,t)|^2$. In fact, it follows from (\ref{34.5}) that $z:=\frac{1}{y}$ is a solution of
  $z'=1-(f(t)+1)z$, $z(0)=1$, and hence given by
  \bas
	z(t)=e^{-\int_0^t f(s)ds - t} + \int_0^t e^{-\int_s^t f(\sigma)d\sigma-(t-s)}ds,\qquad t\in[0,T).
  \eas
  We claim that
  \be{34.6}
	1 \le y(t) \le e^{B+1} \qquad \mbox{for all } t\in (0,T).
  \ee
  %Indeed, to obtain the right inequality we first use (\ref{34.1}) and H\"older's inequality to obtain
  %\bas
	%\int_s^t f(\sigma)d\sigma
	%&\le& (t-s)^\frac{1}{2} \cdot \bigg( \int_s^t \Big( \io |\nabla \ue| \Big)^2 \bigg)^\frac{1}{2} \\
	%&\le& (t-s)^\frac{1}{2} \cdot |\Omega|^\frac{1}{2} \cdot \Big( \int_0^T \io |\nabla \ue|^2 \Big)^\frac{1}{2} \\
	%&\le& (t-s)^\frac{1}{2} \cdot (|\Omega|C_0)^\frac{1}{2}
  %\eas
  %whenever $0\le s\le t< T$.
  To see this, we note that if $t\in (0,T)$ satisfies $t<1$, then (\ref{34.1}) implies
  \bas
	z(t) &\ge& e^{-\int_0^t f(s)ds - t}
	\ge e^{-B - t} \ge e^{-B - 1}, %\qquad 0<t<\min\{1,T\},
  \eas
  whereas if $t\in [1,T)$ then again (\ref{34.1}) shows
  \bas
	z(t)  &\ge& \int_{t-1}^t e^{-\int_s^t f(\sigma)d\sigma-(t-s)}ds \ge \int_{t-1}^t e^{-B-(t-s)}ds \\
	&\ge& \int_{t-1}^t e^{-B - 1}ds = e^{-B - 1}.%, \qquad 1 \le t < T.
  \eas
  This yields the right inequality in (\ref{34.6}), while the left immediately results from an ODE comparison of $y$ with
  $\underline{y}(t) \equiv 1$, because $\underline{y}'+\underline{y}^2-(f(t)+1)\underline{y}=-f(t) \le 0$.
  Consequently, since $\Phi \ge 0$ in $\Omega$, the function $\ou$ defined by (\ref{34.3}) satisfies
  \bas
	\ou(x,0) = M+\Phi(x) \ge M \ge \ue(x,0) \qquad \mbox{for all } x\in\Omega
  \eas
 due to \eqref{34.1}, and on the lateral boundary we have
  \bas
	\ou(x,t) = y(t) \cdot M \ge M \ge \eps \qquad \mbox{for all } x\in\pO \mbox{ and } t\in (0,T).
  \eas
  Moreover,
  \bas
	\ou_t - \ou \Delta \ou - f(t) \cdot \ou
	&=& y' \cdot (M+\Phi) + y^2 \cdot (M+\Phi) - f(t) y \cdot (M+\Phi) \\
	&=& y \cdot (M+\Phi) \\
	&\ge& 0 \qquad \mbox{for all } x\in\Omega \mbox{ and } t \in (0,T),
  \eas
  whence the comparison principle ensures that $\ue \le \ou$ in $\Omega \times (0,T)$. In view of (\ref{34.6}), this entails that
  \bas
	\ue(x,t) \le e^{B+1} \cdot \big(M+\|\Phi\|_{L^\infty(\Omega)}\big)
	\qquad \mbox{for all } x\in\Omega \mbox{ and } t \in (0,t),
  \eas
  so that (\ref{34.2}) is valid upon an obvious choice of $C=C(M,B)$.
\end{proof}
Next, the fact that solutions of (\ref{0eps}) cannot blow up immediately can be turned into a quantitative local-in-time
boundedness estimate in terms of the norm of the initial data in $L^\infty(\Omega) \cap W^{1,2}(\Omega)$.
Moreover, our technique at the same time yields an estimate involving integrals of $u_{\eps t}$ and $\nabla \ue$,
as long as $\ue$ is appropriately bounded.
\begin{lem}\label{lem26}
  i) \ For all $M>0$ there exist $T_1(M)>0$ and $C_1(M)>0$ such that if
  \be{26.00}
	u_{0\eps} \le M \quad \mbox{in } \Omega \qquad \mbox{and} \qquad \io |\nabla u_{0\eps}|^2 \le M
  \ee
  hold for some $\eps \in (\eps_j)_{j\in\N}$, then
  \be{26.01}
	\ue \le C_1(M) \qquad \mbox{in } \Omega \times [0,T_1(M)).
  \ee
 ii) \ For each $M>0$ and $T>0$ there exist $T_2(M)\in(0,T]$ and  $C_2(M)>0$ such that whenever $\eps \in (\eps_j)_{j\in\N}$ is such that
  \be{26.03}
	\ue \le M \qquad \mbox{in } \Omega \times (0,T) \qquad \mbox{and} \qquad \io |\nabla u_{0\eps}|^2 \le M
  \ee
  are satisfied, then
  \be{26.02}
	\int_0^{T_2(M)} \io \frac{u_{\eps t}^2}{\ue} + \sup_{t\in (0,T_2(M))} \io |\nabla \ue(\cdot,t)|^2 \le C_2(M).
  \ee
\end{lem}
\begin{proof}
  i) \ We multiply (\ref{0eps}) by $\frac{u_{\eps t}}{\ue}$ and integrate by parts, use that $u_{\eps t}=0$ on $\pO$, and apply H\"{o}lder's together with Young's inequality to see that
  \bea{26.1}
	\io \frac{u_{\eps t}^2}{\ue} + \frac{1}{2} \frac{d}{dt} \io |\nabla \ue|^2
	&=& \left( \io u_{\eps t} \right) \cdot \rhoeps \left( \io |\nabla \ue|^2 \right) \nn\\
	&\le& \left(  \io \frac{u_{\eps t}^2}{\ue} \right)^\frac{1}{2} \left(  \io \ue \right)^\frac{1}{2}
		 \io |\nabla \ue|^2  \nn\\
	&\le& \frac{1}{2} \io \frac{u_{\eps t}^2}{\ue} + \frac{1}{2} \left(  \io \ue \right) \left(  \io |\nabla \ue|^2 \right)^2
  \eea
  for all $t>0$, because $\rhoeps(z) \le z$ for all $z \ge 0$. Hence,
  \be{26.11}
	\io \frac{u_{\eps t}^2}{\ue} + \frac{d}{dt} \io |\nabla \ue|^2
	\le  \left( \io \ue \right) \left( \io |\nabla \ue|^2 \right)^2.
  \ee
  Using the Poincar\'e inequality% (cf.~(\ref{p}) below)
, we obtain
  \bas
	\io \ue(\cdot,t) \le c_1\cdot \left( \left(\io |\nabla \ue(\cdot,t)|^2\right)^\frac{1}{2} + 1 \right)
  \eas
  with a positive constant $c_1$ independent of $\eps\in(\eps_j)_{j\in\N} \in (0,1)$ and $t>0$. Therefore, (\ref{26.11}) yields
  \be{26.2}
	\io \frac{u_{\eps t}^2}{\ue} + \frac{d}{dt} \io |\nabla \ue|^2
	\le c_1\cdot \left( \left(\io |\nabla \ue|^2\right)^\frac{1}{2} + 1 \right)
	\left( \io |\nabla \ue|^2 \right)^2,
  \ee
  which in particular implies that $y(t):=\io |\nabla \ue(\cdot,t)|^2$ satisfies
  \bas
	y'(t) \le c (\sqrt{y}+1)y^2 \qquad \mbox{for all } t>0.
  \eas
  Hence, if we let $z$ denote the local-in-time solution of
  \bas
	\left\{ \begin{array}{l}
	z'(t)=c(\sqrt{z}+1)z^2, \qquad t>0, \\[1mm]
	z(0)=M,
	\end{array} \right.
  \eas
  with maximal existence time $T_z>0$, then due to (\ref{26.00}) and an ODE comparison we have $y \le z$ in $(0,T_z)$.
  Defining $T_1(M):=\frac{1}{2} T_z$, for instance, we obtain from this that $\io|\nabla \ue(\cdot,t)|^2 \le z(T_1(M))$
  for all $t \in [0,T_1(M))$, whereupon (\ref{26.01}) now results from Lemma \ref{lem34}.\\
  ii) \ If the first inequality in (\ref{26.03}) holds then (\ref{26.11}) entails that $y$ as defined above even
  satisfies the nonlinear ODI
  \bas
	y'(t) \le M|\Omega| y^2 \qquad \mbox{for all } t>0,
  \eas
  whence we have $\io |\nabla \ue(\cdot,t)|^2 \le \frac{1}{M^{-1}-M|\Omega|t}$ for all $t \in (0,T_2)$ with $T_2:=\min\{T, 1/(M^2 |\Omega|)\}$, by the second inequality in (\ref{26.03}).  Inserting this into (\ref{26.2}) again and integrating over $(0,T_2)$ proves (\ref{26.02}).
\end{proof}
When constructing the solution $u$ of \eqref{0} as the limit of solutions $\ue$ of \eqref{0eps}, it will be comparatively easy to obtain the approximation property $\na\ue\to\na u$ in the sense of $L^2_{loc}(\Om\times[0,T))$-convergence. For handling the nonlocal term in the equation, however, it seems appropriate to make sure that also $\io |\na \ue|^2\to \io |\na u|^2$ in $L^1_{loc}([0,T))$.\\
In order to achieve the latter we exclude certain boundary concentration phenomena of $\na \ue$ in the following sense.
\begin{lem}\label{uepsbdry}
 For any $T>0$, $C>0$, $M>0$ and $\delta>0$, there is $K=K(M,C,T,\delta)\subset\subset \Om$ and $\eta>0$ such that whenever $\eps\in (\eps_j)_{j\in\N}$ is such that $\eps<\eta$ and
\begin{equation}\label{eq:bdcondp14}
 \sup_{t\in [0,T]} \io |\nabla \ue(t)|^2 \leq C \quad\mbox{and}\quad \ue\leq M,
\end{equation}
we have
\[
 \int_0^T\int_{\Om\setminus K} |\na \ue|^2<\delta.
\]
\end{lem}

\begin{proof}
For $q\in(0,1)$, we multiply \eqref{0eps} by $\ue^{q-1}$ and integrate by parts to obtain
\[
 \frac1q\frac{d}{dt} \io \ue^q=\intdom \ue^q\dN \ue -\io q\ue^{q-1}|\na\ue|^2+\io\ue^q\;\rhoeps \left( \io |\nabla \ue|^2 \right),
\]
where we can use $\dN \ue\leq 0$ on $\dOm$ and integrate with respect to time to derive
\begin{equation}\label{ps1}
 q\intnT\io\ue^{q-1}|\na\ue|^2\leq-\frac1q\io\ue^q(T)+\frac1q\io\uen^q+\intnT\left(\io\ue^q\io|\na\ue|^2\right)=:C(T)
\end{equation}
for all $\eps>0$ satisfying \eqref{eq:bdcondp14}, which gives control on $|\na\ue|^2$ whereever $\ue$ is small -- which is the case near the boundary, as we ensure next:
In order to lay the groundwork for the corresponding comparison argument, note that by \eqref{eq:bdcondp14},
\bas
 u_{\eps t}=\ue\Delta\ue+\ue\rhoeps\left(\io|\na\ue|^2\right)\leq \ue\Delta\ue+C\ue, \qquad  \ue|_\dOm=\eps, \qquad \ue(0)=\uen.
\eas
Fix $\eta>0$ such that $\frac{(2\eta)^{1-q} C(T)}q < \delta$.
Let $\Phi$ solve \eqref{eq:defPhi}.
Choose $A>C$ such that $A\Phi+\eta>\uen$ for all $0<\eps<\eta$, which is possible due to condition \eqref{ae}. %, which ensured that $\uen-\eps$ can be estimated from above by the distance to $\dOm$.
Then $\ou:=A\Phi+\eta$ satisfies
 \be{eq:uepsbdry_ou}
  \ou_t=0\geq -(A\Phi+\eta)A+(A\Phi+\eta)C=\ou A\Delta\Phi+C\ou=\ou\Delta\ou+C\ou.
 \ee
As long as $\eps<\eta$, also $\ou\bdry\geq \ue\bdry$ holds and furthermore
\[
 \ou(0)\geq \uen.
\]
Therefore, by the comparison principle, we obtain $\ou\geq \ue$.

Now choose $K\subset\subset \Om$ in such a way that
\[
 A\Phi\leq \eta \qquad \mbox{in } \Om\setminus K.
\]
This entails $\ue\leq \ou=A\Phi+\eta \leq 2\eta$
in $\Om\setminus K$. Then
\bgee
  \int_0^T \int_{\Om\setminus K} |\na\ue|^2 &=& \int_0^T \int_{\Om\setminus K} \ue^{q-1} |\na\ue|^2 \ue^{1-q} \\
  &\leq& (2\eta)^{1-q} \int_0^T \int_{\Om\setminus K} \ue^{q-1} |\na\ue|^2 \\
&\leq& (2\eta)^{1-q} \int_0^T \int_{\Om} \ue^{q-1} |\na\ue|^2 \leq \frac{(2\eta)^{1-q} C(T)}{q},
\egee
by virtue of \eqref{ps1}.
\end{proof}
We are now ready to prove that the $\ue$ in fact approach a weak solution of (\ref{0}) that is locally positive in the sense
of Definition \ref{defi44}.
Before we do so, however, we prepare the following estimate for $\ue$ that will be useful in proving assertions about the blow-up behaviour of $u$.
\begin{lem}\label{thm:control_gradient_l2}
 Let  $\Om'\subset\subset\Om$ be a domain with smooth boundary. Assume also that $\phi$ denotes the solution to $-\Delta \phi=1$ in $\Om'$, $\phi|_{\dOm'}=0$.
 Then there exists $C_{\Om'}>0$ such that for each $\eps\in(\eps_j)_{j\in\N}$ and any $t>0$ the solution $\ue$ of \eqref{0eps} satisfies
\begin{align}\label{ps2}
  &\io |\na\ue(\cdot,t)|^2\leq\nn\\ &\io|\na\uen|^2\exp\Bigg[\frac1{2C_{\Om'}}\left(\sup_{\tau\in(0,t)}\io \ue(\tau)\right)\left(\iop \phi\ln\ue(\cdot,t)-\iop \phi\ln\uen+\int_0^t\iop  \ue \right) \Bigg].
\end{align}
\end{lem}
\begin{proof}
As $\uet=0$ on $\dOm$, similarly to \eqref{26.1}, multiplying \eqref{0eps} by $\frac{\uet}{\ue}$ and integrating over $\Om$ yields
\bgee
 \io \frac{\uet^2}{u} &=& \io \uet \Delta \ue + \io \uet \rhoeps\left(\io |\na \ue|^2\right)\\
&=& -\frac12 \ddt \io|\na\ue|^2 + \io \uet \rhoeps\left(\io |\na \ue|^2\right).
\egee
After rearranging, by H\"older's and Young's inequalities and the definition of $\rhoeps$ this entails
\bgee
 \ddt \io|\na\ue|^2&\leq& -2\io\frac{\uet^2}{\ue}+ 2\left[\left(\io\left(\frac{\uet}{\sqrt{\ue}}\right)^2\right)^{\frac12}\left(\io\sqrt{\ue}^2\right)^{\frac12}\right] \rhoeps\left(\io|\na\ue|^2\right)\\
&\leq&-2\io\frac{\uet^2}{\ue}+2\io\frac{\uet^2}{\ue} + \frac12 \io\ue \rhoeps\left(\io|\na\ue|^2\right)^2\\
&\leq& \frac12 \io \ue \rhoeps\left(\io|\na\ue|^2\right)\io|\na \ue|^2\qquad \mbox{on } (0,\infty).
\egee

This looks like a quadratic differential inequality for $y(t)=\io|\na\ue|^2$ and at first does not seem helpful for obtaining an estimate for this quantity. Therefore we shall split the respective quadratic term and apply Gronwall's lemma to
\( y'(t)\leq g(t)y(t)\),
where
\[
 g(t)=\frac12 \io \ue(t) \rhoeps\left(\io|\na\ue(t)|^2\right)\!,
\]
 which leads to
\be{aftergronwall}
 y(t)\leq y(0)\exp\int_0^t g(\tau) d\tau\qquad \mbox{ for all } t>0.
\ee
In this situation, however, we are left with a term $\int_0^t \rhoeps\left(\io|\na\ue|^2\right)$ in the exponent and we prepare an estimate for this in the following way:
With $\phi$ as specified in the hypothesis, we let $C_{\Om'}=\int_{\Om'}\phi>0$.
Multiplication of \eqref{0eps} by $\frac{\phi}{\ue}$ and integrating over $\Om'$ then gives
\bgee
 \ddt\int_{\Om'} \ln\ue\phi&=&\int_{\Om'}\Delta  \ue\phi + \int_{\Om'}\phi\rhoeps\left(\io |\na \ue|^2\right)\\
 &=&\int_{\Om'}\ue\Delta\phi-\int_{\dOm'}\dN\ue\phi+\int_{\dOm'} \ue\dN \phi + C_{\Om'}\rhoeps\left(\io |\na \ue|^2\right) \mbox{ on }(0,\infty).
\egee
Taking into account the definition of $\phi$ and its consequence $\dN\phi|_{\del \Om'} \leq 0=\phi|_{\dOm'}$, we infer that
\bas
\ddt\int_{\Om'}\phi\ln\ue\geq-\int_{\Om'}\ue+C_{\Om'}\rhoeps\left(\io |\na \ue|^2\right) \quad \mbox{ on } (0,\infty).
\eas
Therefore
\bas
 \int_0^t \rhoeps\left(\io |\na \ue|^2\right) \leq \frac1{C_{\Om'}}\left[\int_0^t\int_{\Om'}\ue+\int_{\Om'}\phi\ln\ue(t)-\int_{\Om'} \phi\ln\uen \right]
\eas
for any $t>0$, and we can conclude from \eqref{aftergronwall} that
\[
 \io|\na\ue(t)|^2 \leq \io|\na \uen|^2 \exp\left[\frac1{2C_{\Om'}} \sup_{\tau\in(0,t)} \io u(\tau) \left(\int_0^t\iop \ue+\iop \phi\ln\ue(t)-\iop  \phi\ln\uen \right) \right]
\]
for all $t>0$.
\end{proof}

Another useful piece of information is that a condition like (H3) remains satisfied for any $t>0.$

\begin{lem}\label{lem:normPhiinftyueps}
 Let $T>0$, $M>0$ and $\eps\in(\eps_j)_{j\in\N}$ be such that $\norm[\Phi,\infty]{\uen-\eps}<\infty$. Then any solution $\ue$ of \eqref{0eps} which satisfies
\[
 \io |\na\ue(t)|^2\leq M\qquad\mbox{for any}\quad t\in[0,T]
\]
already fulfils
\[
 \norm[\Phi,\infty]{\ue-\eps}\leq \max\set{M,\norm[\Phi,\infty]{\uen-\eps}}.
\]
\end{lem}
\begin{proof}
 Let $C=\max\smallset{M,\norm[\Phi,\infty]{\uen-\eps}}$ and consider $\ou:= C\Phi+\eps$ with $\Phi$ as in \eqref{eq:defPhi}.
 Then $\ou_t=0\geq (M-C) (C\Phi+\eps) = \ou\Lap \ou+M\ou$,
 whereas $u_{\eps t}=\ue\Lap \ue+\ue \rhoeps\left(\io|\na\ue|^2\right)\leq \ue\Lap \ue +M\ue$.
 Additionally $\ou\bdry=\eps=\ue\bdry$ and $\ou(x,0)-\eps=C\Phi(x)\geq \Phi(x)\norm[\Phi,\infty]{\uen-\eps}\geq \uen(x)-\eps$ and
therefore the comparison principle \cite{Wie} asserts that $\ue\leq \ou$ and hence implies the claim.
\end{proof}
With this information at hand, we can proceed to the proof of convergence of the $\ue$ to a solution of \eqref{0} that still satisfies an inequality like
\eqref{ps2}.
\begin{lem}\label{theo28}
  Suppose that $u_0$ satisfies (H1)-(H3). Then there exists $T>0$ depending on bounds on $\norm[\Liom]{u_0}$ and $\norm[L^2(\Om)]{\na u_0}$ and a locally positive weak solution $u$ of (\ref{0})
  in $\Omega \times (0,T)$. This solution can be obtained as the a.e.~pointwise limit of a subsequence of
  the solutions $\ue$ of (\ref{0eps}) as $\eps=\eps_j \searrow 0$, and for any smoothly bounded subdomain $\Om'\subsubset\Om$ there is $C_{\Om'}>0$ such that
 \begin{align}\label{theestimate}
  &\io |\na u(\cdot,t)|^2\nn\\&\leq \io|\na u_0|^2\exp\Bigg[\frac{1}{2C_{\Om'}}\left(\sup_{\tau\in(0,t)}\io u(\tau)\right)\left(\iop \phi\ln u(\cdot,t)-\iop \phi\ln u_0+\int_0^t\iop  u \right) \Bigg]
 \end{align}
 as well as
 \be{controlsolnPhi}
  \norm[\Phi,\infty]{u(t)}\leq \max \set{\norm[\Phi,\infty]{u_0}, \esssup_{\tau\in(0,t)} \io|\na u(\tau)|^2}
 \ee
 for a.e. $t\in(0,T).$
\end{lem}
\begin{proof}
  We set $M_1:=\max\{\|u_0\|_{L^\infty(\Omega)}+1, \io |\nabla u_0|^2 + 1 \Big\}$ and let $T_1=T_1(M_1)$ and $c_1=C_1(M_1)$ be as in Lemma
  \ref{lem26} i). Then this lemma states that $\ue \le c_1$ in $\Omega \times (0,T_1)$ for all $\eps \in (\eps_j)_{j\in\N}$.
Accordingly, corresponding to $M_2=\max\{ c_1,\io |\nabla u_0|^2+1\}$, Lemma \ref{lem26} ii) provides $T=T_2(M_2)\in(0,T_1)$ and $c_2=C_2(M_2)>0$ such that
  \be{28.4}
	\int_0^{T} \io \frac{u_{\eps t}^2}{\ue} + \sup_{t\in (0,T)} \io |\nabla \ue (\cdot,t)|^2 \le c_2
  \ee
  for all $\eps\in (\eps_j)_{j\in\N}$, which by $\ue\le c_1$ can be turned into a uniform bound on $\norm[L^2(\Om\times(0,T)]{\uet}$, from which it follows by means of the fundamental theorem of calculus that after possibly enlarging $c_2$, we also have
  \be{28.44}
	\|\ue\|_{C^\frac{1}{2}([0,T];L^2(\Omega))} \le c_2
  \ee
  for such $\eps$.\\
  In order to prove a uniform estimate for $\ue$ from below, locally in space, we follow a standard comparison procedure: we pick any smoothly bounded domain
  $\Omega' \subsubset \Omega$ and let $\phi \in C^2(\bar\Omega')$ solve $-\Delta \phi=1$ in $\Omega'$
  with $\phi|_{\pO'}=0$. Then the lower estimates in (\ref{a3}) guarantee that
  \be{28.5}
	u_{0\eps}(x) \ge c_3(\Omega') \phi(x) \qquad \mbox{for all } x\in\Omega'
  \ee
  holds with some appropriately small $c_3(\Omega')>0$. Letting $y(t):=\frac{c_3(\Omega')}{1+c_3(\Omega') t}, t\ge 0$,
  denote the solution of
  $y'=-y^2$ with $y(0)=c_3(\Omega')$, we thus find that $\uu(x,t):=y(t) \phi(x)$ satisfies $\uu \le \ue$ on the parabolic boundary
  of $\Omega' \times (0,\infty)$. Since
  \bas
	\uu_t - \uu \Delta \uu = y' \phi + y^2 \phi = 0 \qquad \mbox{in } \Omega' \times (0,\infty)
  \eas
  and
  \bas
	u_{\eps t} - \ue \Delta \ue = \ue \cdot \rhoeps \Big( \io |\nabla \ue|^2 \Big) \ \ge 0
	\qquad \mbox{in } \Omega \times (0,\infty),
  \eas
  we conclude from the comparison principle (see \cite{Wie} for an adequate version) that $\uu \le \ue$ and thus,
  in particular, that for each compact $K\subset \Omega$ and $T'>0$ there exists a suitably small $c_4(K,T')>0$ such that
  \be{28.6}
	\ue \ge c_4(K,T') \qquad \mbox{in } K \times (0,T')
  \ee
  holds for all $\eps\in(\eps_j)_{j\in\N}$. 
  Now the estimate $\ue\leq c_1$, (\ref{28.4}), (\ref{28.44}) and (\ref{28.6}) along with standard compactness arguments allow us to extract
  a subsequence $(\eps_{j_k})_{k\in\N}$ of $(\eps_j)_{j\in\N}$ and a function $u: \ \Omega \times [0,T] \to \R$ % as well as $z\in L^\infty((0,T);L^2(\Om)$
such that
  \begin{eqnarray}
	\ue \to u & & \mbox{in } \ C^0([0,T);L^2(\Omega)) \quad \mbox{and a.e.~in } \Omega \times (0,T), \label{28.7} \\
	\nabla \ue \wto \nabla u & & \mbox{in } \ L^2_{loc}(\bar\Omega \times [0,T)) \qquad \mbox{and} \label{28.8} \\
	u_{\eps t} \wto u_t & & \mbox{in } \ L^2(\Omega \times (0,T)) \label{28.9}%\qquad \mbox{and} %\\
%	\na u_{\eps} \wto^* z & & \mbox{in } \ L^\infty((0,T),L^2(\Om) \label{28.95}
  \end{eqnarray}
  as $\eps=\eps_{j_k} \searrow 0$. From (\ref{28.7}), the inequality $\ue\leq C_1$ and (\ref{28.6}), we know that
  $u\le c_1$ a.e.~in $\Omega \times (0,T)$ and $u \ge c_4(K,T)$ a.e.~in $K \times (0,T)$ whenever $K \subsubset \Omega$. Moreover, since $\ue-\eps$ vanishes on $\pO$, (\ref{28.8}) implies that
  $u\in L^2 ((0,T);W_0^{1,2}(\Omega))$, so that $u$ fulfills all regularity and positivity properties required for a
  locally positive weak solution in $\Omega \times (0,T)$ in the sense of Definition \ref{defi44}.\\
  In order to verify that $u$ is a weak solution of (\ref{0}) it thus remains to check (\ref{0w}). To prepare this,
  we claim that in addition to (\ref{28.8}), we also have the strong convergence properties
  \be{28.99}
	\nabla \ue \to \nabla u \qquad \mbox{in } L^2_{loc}(\Omega \times [0,T])
	\qquad \mbox{and a.e.~in } \Omega \times (0,T)
  \ee
  as well as
  \be{28.999}
	\io |\nabla \ue(x,\cdot)|^2 dx \to \io |\nabla u(x,\cdot)|^2 dx \qquad \mbox{in } L^1((0,T))
  \ee
  as $\eps=\eps_{j_k}\searrow 0$. % Note that \eqref{28.99} also entails $z=\na u$ in \eqref{28.95}.
  To see (\ref{28.99}), we let $K \subsubset \Omega$ %and $T' \in (0,T)$
  be given and fix a nonnegative $\psi \in C_0^\infty(\Omega)$ such
  that $\psi\equiv 1$ in $K$. Then
  \bea{28.100}
	\int_0^{T} \int_K |\nabla \ue - \nabla u|^2
	&\le& \int_0^{T} \io |\nabla \ue-\nabla u|^2 \psi \nn\\
	&=& \int_0^{T} \io \nabla (\ue-u) \cdot \nabla \ue \cdot \psi
		+ \int_0^{T} \io \nabla u \cdot \nabla (\ue-u) \cdot \psi \nn\\
	&=:& I_1(\eps)+I_2(\eps)\qquad \mbox{ for all }\eps\in(\eps_j)_{j\in\N},
  \eea
  where $I_2(\eps) \to 0$ as $\eps=\eps_{j_k} \searrow 0$ by (\ref{28.8}). Using the equation for $\ue$, however, after an
  integration by parts we find that
  \bas
	I_1(\eps) &=& - \int_0^{T} \io (\ue-u)\Delta \ue \cdot \psi - \int_0^{T} \io (\ue-u) \nabla \ue \cdot \nabla \psi \nn\\
	&=& - \int_0^{T} \io (\ue-u) \cdot \frac{u_{\eps t}}{\ue} \cdot \psi
		+ \int_0^{T} \io (\ue-u) \cdot \rhoeps \Big( \io |\nabla \ue|^2 \Big) \cdot \psi \nn\\
	& & - \int_0^{T} \io (\ue-u) \nabla \ue \cdot \nabla \psi \nn\\
	&=:& I_{11}(\eps)+I_{12}(\eps)+I_{13}(\eps)\qquad \mbox{ for all } \eps\in(\eps_j)_{j\in\N}.
  \eas
  Due to (\ref{28.7}) and (\ref{28.8}), we have $I_{13}(\eps) \to 0$, and (\ref{28.7}) together with (\ref{28.4}) and
  H\"older's inequality imply that
  \bas
	|I_{12}(\eps)| &\le& \Big( \int_0^{T} \io (\ue-u)^2 \Big)^\frac{1}{2} \cdot
	\Big[ \int_0^{T}\! \Big(\io |\nabla \ue|^2 \Big)^2\Big]^\frac{1}{2} \cdot \|\psi\|_{L^2(\Omega)} \to 0
  \eas
  as $\eps=\eps_{j_k}\searrow 0$, where we again have used the fact that $\rhoeps(z)\le z$ for any $z\ge 0$ and all
  $\eps \in (\eps_j)_{j\in\N}$.
  We now use H\"older's inequality and the local lower estimate (\ref{28.6}), which in conjunction with (\ref{28.4}) yields
  \bas
	|I_{11}(\eps)| &\le& \Big( \int_0^{T} \io \frac{u_{\eps t}^2}{\ue} \Big)^\frac{1}{2} \cdot
	\Big( \int_0^{T'} \io \frac{(\ue-u)^2}{\ue} \cdot \psi^2 \Big)^\frac{1}{2} \nn\\
	&\le& c_2^\frac{1}{2} \cdot \frac{\|\psi\|_{L^\infty(\Omega)}}{(c_4(\supp \psi, T))^\frac{1}{2}} \cdot
	\Big( \int_0^{T} \io (\ue-u)^2 \Big)^\frac{1}{2} \to 0
  \eas
  as $\eps=\eps_{j_k}\searrow 0$, by \eqref{28.7}. Altogether, we obtain that $I_1(\eps) \to 0$ and hence, by (\ref{28.100}), that
  $\nabla \ue \to \nabla u$ in $L^2(K \times (0,T))$ as $\eps=\eps_{j_k}\searrow 0$ for arbitrary $K\subsubset \Omega$.\\
Having thus proved (\ref{28.99}), with the aid of Lemma \ref{uepsbdry} we obtain (\ref{28.999}) as a straightforward consequence:

Given $\delta>0$, we let $K=K(c_1,c_2,T,\frac{\delta}{4})$ and $\eta>0$ be the set and the constant provided by Lemma \ref{uepsbdry}, and employ the convergence asserted by \eqref{28.8} to choose $k_0\in\N$ such that for all $k,l>k_0$ we have $\int_0^{T}\int_K ||\na u_{\eps_k}|^2-|\na u_{\eps_l}|^2|\leq \frac{\delta}{2}$.
Then for all $k,l>k_0$,
\bas
 \int_0^{T}\left\lvert\io|\na u_{\eps_k}|^2-\io|\na u_{\eps_l}|^2\right\rvert %&=&\int_0^{T'}\left(\int_K |\na u_{\eps_k}|^2-|\na u_{\eps_l}|^2 +\int_{\Om\setminus K} |\na u_{\eps_k}|^2-\int_{\Om\setminus K} |\na u_{\eps_l}|^2 \right)^2\\
 &\leq& \int_0^{T}\int_K \left\lvert|\na u_{\eps_k}|^2-|\na u_{\eps_l}|^2\right\rvert + \int_0^{T}\int_{\Om\setminus K} |\na u_{\eps_k}|^2\\
 && + \int_0^{T} \int_{\Om\setminus K} |\na u_{\eps_l}|^2 \\
&\leq& \frac\delta2+\frac\delta4+\frac\delta4
\eas
and thanks to the completeness of $L^2((0,T))$ %and arbitrariness of $T'\in(0,T)$,
we obtain \eqref{28.999}.
  We can now proceed to verify that (\ref{0w}) holds for all $\varphi \in C_0^\infty(\Omega \times (0,T))$. To this end,
  we multiply (\ref{0eps}) by $\varphi\in C_0^\infty(\Om\times(0,T))$ and integrate to obtain
  \bas
	\int_0^T \io u_{\eps t} \varphi + \int_0^T \io |\nabla \ue|^2 \varphi + \int_0^T \io \ue \nabla\ue \cdot \nabla\varphi
	= \int_0^T \io \ue \cdot \rhoeps \Big( \io |\nabla\ue|^2 \Big) \cdot \varphi.
  \eas
  Here, as $\eps=\eps_{j_k}\searrow 0$ we have
  \bas
	\int_0^T \io u_{\eps t} \varphi \to \int_0^T \io u_t \varphi
  \eas
  by (\ref{28.9}), whereas (\ref{28.99}) and (\ref{28.7}) allow us to conclude that
  \bas
	\int_0^T \io |\nabla\ue|^2 \varphi \to \int_0^T \io |\nabla u|^2 \varphi
  \eas
  and
  \bas
	\int_0^T \io \ue\nabla\ue \cdot \nabla\varphi \to \int_0^T \io u\nabla u \cdot \nabla\varphi,
  \eas
  because $\varphi$ vanishes near $\pO$ and near $t=T$. Finally,
  \bas
	\int_0^T \io \ue \cdot \rhoeps \Big(\io |\nabla\ue|^2 \Big) \cdot \varphi
	\to \int_0^T \io u \Big( \io |\nabla u|^2 \Big) \cdot \varphi
  \eas
  because of (\ref{28.7}), (\ref{28.999}) and the fact that $\rhoeps(z)\to z$ for all $z \ge 0$ as $\eps\searrow 0$.
  We thereby see that (\ref{0w}) holds and thus infer that $u$ in fact is a weak solution of (\ref{0}) in
  $\Omega \times (0,T)$.
The inequality \eqref{theestimate} results from Lemma \ref{thm:control_gradient_l2} and the convergence statements.
%\eqref{28.6} ensures \ue>C>0 in $\Om'$, hence convergence in L^1 implies convergence of ln in L^1; furthermore convergence of initial data and convergence in C((0,t),L^2(\Om)).
The estimate \eqref{controlsolnPhi} results from Lemma \ref{lem:normPhiinftyueps}:
By \eqref{28.4} and \eqref{ae} we have the necessary bounds on gradient and initial value, independent of $\eps\in(\eps_j)_{j\in\N}$. Furthermore, for any $t\in[0,T]$ we can find a subsequence $(\eps_{j_k})_{k\in\N}$ of $(\eps_j)_{j\in\N}$ such that
\[
 \frac{u_{\eps_{j_k}}(t)-\eps}{\Phi} \wto^* \frac {u(t)} \Phi \qquad\mbox{in } L^\infty(\Om)
\]%limits coincide: already have ae (of \ue), hence ae of (\ue-\eps)/\Phi, tog. with L^\infty-bound + LDKS
and finally the same bound as in Lemma \ref{lem:normPhiinftyueps} holds for $u(t)$ because
\begin{align*}
 \norm[\Phi,\infty]{u(t)} =& \norm[\infty]{\frac{u(t)}\Phi} \leq \liminf_{\eps=\eps_{j_k}\searrow0} \norm[\infty]{\frac{u_\eps (t)}\Phi} \\
\leq& \liminf_{\eps=\eps_{j_k}\searrow0} \max\set{\sup_{0<\tau<t} \io|\na \ue(\tau)|^2,\norm[\Phi,\infty]{\uen-\eps}}\\
\leq& \liminf_{\eps=\eps_{j_k}\searrow0} \max\set{\sup_{0<\tau<t} \io|\na \ue(\tau)|^2,\norm[\Phi,\infty]{u_0}+\eps}\\
\leq& \max\set{\esssup_{0<\tau<t} \io |\na u(\tau)|^2,\norm[\Phi,\infty]{u_0}},
\end{align*}
where for the last inequality we relied on the pointwise a.e. convergence of $\io |\na\ue|^2$ in $(0,T)$, due to \eqref{28.999} valid along a subsequence.
\end{proof}
We are now in the position to prove Theorem \ref{thm:Tmax}, which asserts the existence of a locally positive weak solution and $\Tmax\in(0,\infty]$ such that the solution blows up at $\Tmax$ or exists globally.
\begin{proof}[Proof of Theorem \ref{thm:Tmax}]
%Let $L$ denote the number in condition (H3) on the initial data.
By Lemma \ref{theo28} there exists $T>0$ such that \eqref{0} possesses a locally positive weak solution $u$ on $\Om\times(0,T)$, which satisfies \eqref{leq} and \eqref{leqnormphi} for a.e. $t\in(0,T)$.
Hence, the set
\begin{align*}
 S:=\Bigg\{\Ttilde>0\bigg| \mbox{there exists a locally}&\mbox{ positive solution } u\; \mbox{to \eqref{0} on } \Om\times(0,\Ttilde)\;\\
&\mbox{satisfying \eqref{leq} and \eqref{leqnormphi} for a.e. }t\in(0,\Ttilde)\Bigg\}
\end{align*}
is not empty and
\[
 \Tmax=\sup S \in(0,\infty]
\]
is well-defined.
Assume that $\Tmax<\infty$ and $\limsup_{t\nearrow \Tmax} \norm[\Liom]{u(\cdot,t)}<\infty$.

This implies the existence of a constant $M>0$ such that $u\leq M$ and hence, due to \eqref{leq}, also that there is $C>0$ with $\io |\na u|^2\leq C$.
Lemma \ref{theo28} provides $T>0$ such that for any initial data $u_0$ satisfying $u_0\leq M$, $\io|\na u_0|^2\leq C$, a locally positive weak solution existing on $\Om\times (0,T)$ can be constructed.

Choose $t_0\in (\Tmax-\frac T2, \Tmax)$ such that $u(x,t_0)\leq M$ and $\io|\na u(x,t_0)|^2\leq C$
and such that $u$ satisfies \eqref{leq} and \eqref{leqnormphi} at $t=t_0$.

Let $v$ denote the corresponding solution with initial value $u(\cdot,t_0)$
and define
\[
 \uhat(x,t)=\begin{cases}
       u(x,t),& x\in\Om, t<t_0\\
       v(x,t-t_0),& x\in\Om, t\in (t_0,t_0+T).
      \end{cases}
\]
Then $\uhat$ is a solution of \eqref{0}, and \eqref{leq} and \eqref{leqnormphi} obviously hold for a.e. $t\in(0,t_0)$, whereas for $t\in(t_0,t_0+T)$ we have
\begin{align*}
&\io|\na\uhat(\cdot,t)|^2 \\ \leq&  \io|\na u(t_0)|^2\times\\&\times \exp\Bigg[\frac{1}{2C_{\Om'}}\left(\sup_{\tau\in(t_0,t_0+T)}\io \uhat(\cdot,\tau)\right)\left(\iop \phi\ln \uhat(\cdot,t)-\iop \phi\ln u(\cdot,t_0)+\int_{t_0}^{t_0+T} \iop  \uhat \right) \Bigg]\\
\leq&  \io|\na u_0|^2\exp\Bigg[\frac{1}{2C_{\Om'}}\left(\sup_{\tau\in(0,t_0)}\io u(\cdot,\tau)\right)\left(\iop \phi\ln u(\cdot,t_0)-\iop \phi\ln u_0+\int_0^{\cdot,t_0}\iop  u \right) \Bigg]\times\\
&\qquad \times\exp\Bigg[\frac{1}{2C_{\Om'}}\left(\sup_{\tau\in(t_0,t)}\io \uhat(\cdot,\tau)\right)\left(\iop \phi\ln \uhat(\cdot,t)-\iop \phi\ln u(\cdot,t_0)+\int_{t_0}^t\iop  \uhat \right) \Bigg]\\
\leq&  \io|\na u_0|^2\exp\Bigg[\frac{1}{2C_{\Om'}}\left(\sup_{\tau\in(0,t)}\io u(\cdot,\tau)\right)\cdot\\&\;\cdot\Big(\iop \phi\ln u(\cdot,t_0)-\iop \phi\ln u_0+\int_0^{t_0}\iop  u + \iop \phi\ln \uhat(\cdot,t)-\iop \phi\ln u(\cdot,t_0)+\int_{t_0}^t\iop  \uhat \Big) \Bigg]\\
=&  \io|\na u_0|^2\exp\Bigg[\frac{1}{2C_{\Om'}}\left(\sup_{\tau\in(0,t)}\io u(\cdot,\tau)\right)\left(\iop \phi\ln \uhat(\cdot,t) -\iop \phi\ln u_0+\int_0^{t}\iop  u  \right) \Bigg].
\end{align*}
Also, for a.e. $t\in(0,t_0+T)$,
\bas
 \norm[\Phi,\infty]{\uhat(\cdot,t)}\leq& \max\set{\norm[\Phi,\infty]{u(\cdot,t_0)},\sup_{\tau\in(t_0,t)} \io |\na \uhat(\cdot,\tau)|^2} \\
\leq&\max\set{\max\set{\norm[\Phi,\infty]{u_0}, \sup_{\tau\in(0,t_0)} \io|\na u(\cdot,\tau)|^2},\sup_{\tau\in(t_0,t)} \io |\na \uhat(\cdot,\tau)|^2} \\
\leq&\max\set{\norm[\Phi,\infty]{u_0}, \sup_{\tau\in(0,t)}\io |\na \uhat(\cdot,\tau)|^2}.
\eas
Thus $\uhat$ is defined on $(0,\Tmax+\frac{T}2)$, contradicting the definition of $\Tmax$.
\end{proof}

As a direct consequence of \eqref{leqnormphi} we obtain that finite-time gradient blow-up cannot occur.
More precisely, we have the following.
\begin{cor}\label{ngb}
 Let $u$ and $\Tmax$ be as given by Theorem \ref{thm:Tmax}.\\
 If $\limsup_{t\nearrow \Tmax} \norm[\Liom]{u(\cdot,t)}=\infty,$ then also
\[
 \limsup_{t\nearrow \Tmax} \io |\na u(x,t)|^2 dx= \infty.
\]
\end{cor}

Combining now  Corollary \ref{ngb} with the estimate \eqref{leq}, we can conclude that if finite-time $L^\infty-$blow-up occurs then also $L^1-$blow-up takes place at the same finite time.
\begin{cor}\label{l1blowup}
Let $u$ and $\Tmax$ be as given by Theorem \ref{thm:Tmax}.\\
 If $\limsup_{t\nearrow \Tmax} \norm[\Liom]{u(\cdot,t)}=\infty,$ then also
\[
 \limsup_{t\nearrow \Tmax} \io u(x,t) dx = \infty.
\]
\end{cor}

\section{Total mass. Proof of Theorem \ref{thm:longterm}}\label{s3}
Let $u$ be a solution of \eqref{0} on $[0,T]$.
Consider its mass
\be{defy}
 y(t)=\io u(x,t)\,dx,\qquad t\in[0,T),
\ee
and note that \eqref{defy} defines a continuous function on $[0,T]$. Indeed, we have the following.
\begin{lem}\label{totalmass}
 For any weak solution $u$ of \eqref{0} on $[0,T]$, \eqref{defy} defines an absolutely continuous function $y\colon[0,T]\to \R$ that satisfies
\begin{equation}\label{nk1}
 y'(t)=(y(t)-1)\io|\na u(x,t)|^2\, dx
\end{equation}
 for almost every $t\in(0,T)$.
\end{lem}
\begin{proof}
We will show that whenever $0<s<t<T$,
\be{yeq}
 y(t)-y(s) = \int_s^t \left((y(\tau)-1)\io |\na u(x,\tau)|^2\, dx\right) d\tau,
\ee
where absolute continuity follows from the representation as integral and the assertion about the derivative is a direct consequence of division by $t-s$ and passing to the limit $s\to t$.

Let $0<s<t<T$ and $0<\delta<\min\set{s,T-t}$.
Define the function $\chi\colon \R\to\R$ by setting:
\[
 \chi(\tau)=\begin{cases}
0,& \tau<s-\delta,\\
1+\frac{\tau-s}\delta,&s-\delta\leq\tau<s,\\
1, & s\leq\tau<t,\\
             1-\frac{\tau-t}\delta,&t\leq \tau<t+\delta,\\
	      0,&\tau\geq t+\delta.
\end{cases}
\]
% For reference: the standard approximation arguments referred to below:
%
% \begin{rem}\label{rem:othertestfcts}
% Let $T'<T$, let $v\in L^2((0,T');W^{1,2}_0(\Om))\cap L^\infty(\Om\times(0,T'))$ and $v=0$ on $\Om\times(T',T)$. Then, for any weak solution, \eqref{0w} holds for $\phii=v$ as well.
% Let $(\phii_n)_{n\in\N}\subset C_0^\infty(\Om\times(0,T))$ be such that $\phii_n\to v$ in $L^2((0,T');W^{1,2}(\Om))$ and $\norm[L^\infty(\Om\times(0,T'))]{\phii_n}\leq \norm[L^\infty(\Om\times(0,T'))]{v}+1$ for all $n\in\N$. Then, obviously, $u_t\phii_n\to u_t v$ in $L^1(\Om\times(0,T'))$, $u\nabla \phii_n\to u\nabla v$ in $L^2(\Om\times(0,T'))$ and $\io u\phii_n \to \io uv$ in $L^2((0,T'))$.
% Moreover, $\phii_n\nabla u\to v\nabla u$ in $L^1(\Om\times(0,T'))$ and by boundedness of $v$ and $\norm[L^\infty]{\phii_n}$ and a.e. convergence along a subsequence in combination with Lebesgue's theorem, $\phii_n\nabla u\to v\nabla u$ in $L^2(\Om\times(0,T'))$. Hence \eqref{0w} holds for $v$.
% \end{rem}
Then, according to standard approximation arguments, $\phii(x,t):=\chi(t)$ defines an admissible test function for \eqref{0vw} and we obtain
  \bas
	-\frac1\delta \int_{s-\delta}^s \io u + \frac{1}{\delta} \int_{t}^{t+\delta} \io u + \int_{s-\delta}^{t+\delta} \io |\nabla u|^2 \varphi
	= \int_{s-\delta}^{t+\delta} \Big( \io u\varphi \Big) \cdot \Big( \io |\nabla u|^2 \Big).
  \eas
Since $u\in C_{loc}([0,T),L^2(\Om))$, we have
\[
 \frac{1}{\delta} \int_{t}^{t+\delta} \io u\to u(t) \quad \mbox{and} \quad \frac{1}{\delta} \int_{s-\delta}^{s} \io u\to u(s)
\]
as $\delta\searrow 0$.

Also by Lebesgue's dominated convergence theorem,
\[
 \int_{s-\delta}^{t+\delta} \io |\nabla u|^2 \varphi \to  \int_s^{t} \io |\nabla u|^2
\]%\phii\to 1 / 0 pw., |\na u|^2 \in L^1 Majorant
and
\[
 \int_{s-\delta}^{t+\delta} \Big( \io u\varphi \Big) \cdot \Big( \io |\nabla u|^2 \Big) \to \int_s^{t} \Big( \io u \Big) \cdot \Big( \io |\nabla u|^2 \Big)
\]%\phii-> pw; u\io|\na u|^2 \in L^1
as $\delta\searrow 0$. Hence, \eqref{yeq} holds.
\end{proof}

This lemma is the main ingredient in the following proof of Theorem \ref{thm:longterm}:

\begin{proof}[Proof of Theorem \ref{thm:longterm}]
(i) \ In the case of subcritical initial mass Lemma \ref{totalmass} shows that y as defined in \eqref{defy} is decreasing, which by Corollary \ref{l1blowup} entails global existence,
and from the nonnegativity of $y$ we derive that $y(t)\to c$ as $t\to\infty$ for some $c\geq 0$.
Note that Poincar\'e's and H\"older's inequalities imply that for some $C_P>0$ we have
\[
  	\io |\na u|^2 \,dx \geq \frac1{C_P} \io u^2\,dx \geq \frac1{C_P|\Om|} \left(\io u \,dx\right)^2
	= \frac1{C_P|\Om|} y^2 \qquad \mbox{ on } (0,\infty),
\]
and hence Lemma \ref{totalmass}, due to the negativity of $y(t)-1,$ entails that
\[
 	y'(t) \leq (y(t)-1) \frac1{C_P|\Om|}\, y^2(t) \leq - \frac{1-y(0)}{C_P|\Om|}\, y^2(t)
	\leq - \frac{1-y(0)}{C_P|\Om|} \, c^2
\]
for almost every $t>0$.
This would lead to a contradiction to the nonnegativity of $y(t)$ if $c$ were positive, whence actually $c=0$.\abs
(ii) \
If $\io u_0=1$, then Lemma \ref{totalmass} implies that
\[
	y(t)-1=\int_0^t\left[(y(s)-1)\io|\na u(x,s)|^2\,dx\right]\, ds,
\]
and by virtue of Gronwall's lemma we conclude $y(t)-1\equiv 0$ throughout the time interval on which the solution exists, which combined with Corollary \ref{l1blowup} also implies global existence.\abs
(iii) \
In the case when the total mass is supercritical initially,
Lemma \ref{totalmass} entails that $y$ is nondecreasing, and again Poincar\'e's and H\"older's inequalities imply that
\[
  	y'(t)\geq \frac{y(0)-1}{C_P|\Om|} y^2(t) \qquad \mbox{for a.e. } t\in[0,\Tmax)
\]
with some $C_P>0$.
Now let $z$ denote the solution to
\[
 	z'(t)=\frac{y(0)-1}{C_P|\Om|}z(t)^2,\;z(0)=z_0,
\]
for some $1<z_0<y(0)$, %0<... would be enough,
defined up to its maximal existence time $T_0>0$. Then $T:=\Tmax<T_0$, because $y\geq z$,
and the assertion follows by Theorem \ref{thm:Tmax} in combination with Corollary \ref{l1blowup}.
\end{proof}
\mysection{Global blow-up. Proof of Theorem \ref{thm:blowupset}}\label{s4}
We proceed to prove that blow-up of our solutions always occurs globally, as stated in Theorem
\ref{thm:blowupset}.
\begin{proof}[Proof of Theorem \ref{thm:blowupset}]
Assume to the contrary that the closed set $\mathcal{B}$ is strictly contained in $\Ombar$. Then there exists a smoothly bounded subdomain $\Om'\sub\Om\setminus\mathcal{B}$ such that $u$ is bounded in $\Om'\times(0,\Tmax)$. Let $\phi$ be a solution to $-\Lap \phi=1$ in $\Om'$, $\phi=0$ on $\dOm'$.

Consider $T'<\Tmax$. Due to the local positivity of $u$ we have $\frac\phi u\in L^\infty(\Om\times(0,T'))$ and $\na \frac\phi u=\frac{\na \phi}u-\frac\phi{u^2}\na u\in L^2(\Om'\times(0,T'))$ and hence $\frac{\phi}u\in L^2((0,T'), W_0^{1,2}(\Om'))\cap L^\infty(\Om\times(0,T'))\subset L^2((0,T'), W_0^{1,2}(\Om))\cap L^\infty(\Om\times(0,T'))$. Therefore, it can readily be verified by approximation arguments that it is  possible to use $\phii=\frac{\phi}u$ as a test function in \eqref{0w}, which then leads to
% By Remark \ref{rem:othertestfcts}, it is then possible to use $\phii=\frac{\phi}u$ as a test function in \eqref{0w}, because $\frac\phi u\in L^\infty(\Om\times(0,T'))$ and $\na \frac\phi u=\frac{\na \phi}u-\frac\phi{u^2}\na u\in L^2(\Om'\times(0,T'))$ due to the local positivity of $u$, and hence $\frac{\phi}u\in L^2((0,T'), W_0^{1,2}(\Om'))\cap L^\infty(\Om\times(0,T'))\subset L^2((0,T'), W_0^{1,2}(\Om))\cap L^\infty(\Om\times(0,T'))$.
\bas
	\int_0^t \iop \frac{u_t}u \phi \,dx\,ds+ \int_0^t \iop \nabla u \cdot \nabla \phi \,dx\,ds
	=\int_0^t \left( \iop \phi \,dx \right) \cdot \left( \io |\nabla u|^2 \,dx\right)\,ds
\eas
for any $t\in(0,\Tmax)$.
Hence, with $C_{\Om'}:=\iop \phi$ and because of $\partial_\nu \phi\big\rvert_{\dOm'}\leq 0$,
\bas
\iop \phi \ln u(t)\,dx - \iop \phi \ln u_0\, dx - \int_0^t \iop  u \cdot \Lap \phi \,dx\,ds
	\geq C_{\Om'} \int_0^t  \io |\nabla u|^2\,dx\,ds,
\eas
that is
\be{preparationglobalblowup}
\int_0^t \iop  u \,dx\,ds + \iop \phi \ln u(t)\,dx - \iop \phi \ln u_0\,dx \geq C_{\Om'} h(t),
\ee
where $h(t):=\int_0^t  \io |\nabla u(x,s)|^2\,dx\,ds$ and where -- due to the choice of $\Om'$ -- the left hand side is bounded from above.

On the other hand, from Lemma \ref{totalmass} we know that
\[
 \frac{y'(t)}{y(t)-1} = \io |\na u|^2\,dx
\]
for $y(t)=\io u(x,t)\,dx$. Therefore
\[
 h(t)= \int_0^t \io |\na u|^2 \,dx\,ds= \int_0^t \frac{y'(\tau)}{y(\tau)-1}\,ds = \ln(y(t)-1)-\ln(y(0)-1)=\ln \frac{\io u(x,t)\,dx-1}{\io u_0\,dx-1}
\]
and, by Theorem \ref{thm:longterm} (iii), $\limsup_{t\nearrow \Tmax} h(t)=\infty$, contradicting the boundedness of the left hand side of \eqref{preparationglobalblowup}.
\end{proof}

We have seen that the question of global existence versus blow-up of solutions to \eqref{0} is intimately connected with the size of the initial data. If $\io u_0>1$, the solution blows up globally; if $\io u_0<1$, we have proven convergence towards $0$. The missing case of solutions emanating from initial data with unit mass must exhibit a behaviour different from either, as Theorem \ref{thm:longterm} (ii) shows. For a study of these solutions, which are actually very important for the described replicator dynamics model, we refer the reader to the forthcoming article \cite{La}.

\mysection{Appendix A: Modelling background}
{\it Evolutionary game dynamics} is a major part of modern game theory. It was appropriately fostered by evolutionary biologists such as W. D. Hamilton and J. Maynard Smith (see  \cite{DR98} for a collection of survey papers and \cite{Si93} for a popularized account) and it actually brought a conceptual revolution to the game theory analogous with the one of population dynamics in biology. The resulting population-based approach has also found many applications in non-biological fields like economics or learning theory and introduces a significant enrichment of {\it classical} game theory which focuses on the concept of a rational individual.

The main subject of evolutionary game dynamics is to explain how a population of players
update their strategies in the course of a game according to the strategies' success. This
contrasts with classical noncooperative game theory that analyzes
how rational players will behave through static solution concepts
such as the Nash Equilibrium (NE) (i.e., a strategy choice for each
player whereby no individual has a unilateral incentive to change
his or her behaviour).

As Hofbauer and Sigmund \cite{hs03} pointed out,
strategies with high pay-off will spread within the population through learning, imitation or inheriting processes  or even by  infection. The
pay-offs depend on the actions of the co-players, i.e. the frequencies in which the various strategies appear, and since these
frequencies change according to the pay-offs, a feedback loop appears. The dynamics of this feedback loop will determine the long time progress of the game and its investigation is exactly the course of evolutionary game theory.

According to the extensive survey paper \cite {hs03} there is a variety of different dynamics in evolutionary game theory: replicator dynamics, imitation dynamics, best response dynamics, Brown-von Neumann-Nash dynamics e.t.c.. However, the dynamics most widely used and studied in the literature on evolutionary game theory are the replicator dynamics which was introduced in \cite{TaJ78} and baptised in \cite{ScS83}. Such kind of dynamics illustrate  the idea that in a dynamic process of evolution a strategy  should increase in frequency if it is a successful strategy in the sense that individuals playing this strategy obtain a higher than average payoff.
%One particularly popular update dynamics model is the replicator dynamics scheme. According to this scheme, the
%population updates its strategies by adjusting the logarithmic rate of change of the distribution of strategies in proportion to
%the difference between the actual payoff for a given strategy profile and the average payoff.
%Heuristically, replicator dynamics approach, see also below, increases the frequence of strategies have a higher expected payoff than the mean payoff.

Let us consider a game with $m$ discrete pure strategies, forming the strategy space $S=\{1,2,...,m\},$ and corresponding frequencies $p_i(t), i=1,2,...,m,$ for any $t\geq 0.$ (Alternatively  $S$ could be considered as the set of different states (genetic programmes) of a biological population). The frequency (probability) vector $p(t)=(p_1(t),p_2(t),...,p_m(t))^{T}$ belongs to the invariant simplex
\[S(m)=\left\{y=(y_1 ,y_2,...,y_m)^{T}\in \R^m: y_i\geq 0, i=1,2,...,m\quad\mbox{and}\quad \sum_{i=1}^m y_i=1\right\}.\]
The game is actually determined by the pay-off  matrix $A=(a_{ij})$, which is a real $m\times m$
symmetric matrix. Pay-off means expected gain, and if an individual  plays strategy $i$ against another individual
following strategy $j,$ then the pay-off to $i$ is defined to be $a_{ij}$ while the pay-off
to $j$  is $a_{ji}.$ For symmetric games matrix $A$ is considered to be symmetric.
%Let us denote by $A=(a_{ij})$ the $m\times m$ payoff matrix of the game, where $a_{ij}$ stands for the payoff (on some evolutionary scale) for player $1$ if she follows strategy $i$ whereas player $2$ follows strategy $j$
(In the case of a biological population  pay-off represents fitness,
or reproductive success.)

Then the expected pay-off for an individual playing strategy $i$ can be expressed as
\[
(A\cdot p(t))_i=\sum_{j=1}^m a_{ij} p_j(t),
\]
whereas the average pay-off over the whole population is given by
\[(p(t)^{T}\cdot A\cdot p(t))=\sum_{i=1}^m\sum_{j=1}^m a_{ij} p_i(t) p_j(t).\]
Consider that our game is symmetric with infinitely many players (or that the biological population is infinitely big and its generations blend continuously to each other) then we obtain that $p_i(t)$ evolve as differentiable functions. Note that the rate of increase of the per capita rate of growth $\dot{p}_i/p_{i}$  of strategy (type) $i$ is a measure of its evolutionary success; here $\dot{p}_i$ stands for the time derivative of $p_i.$ A reasonable assumption, which is also in agreement with the basic tenet of Darwinism, is that the per capita rate of growth (i.e. the logarithmic derivative) $\dot{p}_i/p_{i}$ is given by the difference between the pay-off for strategy (type) $i$ and the average pay-off. This yields the {\it the replicator dynamical system},
\begin{equation}\label{rd1}
\frac{dp_i}{dt}=\left(\sum_{j=1}^m a_{ij} p_j(t)-\sum_{i=1}^m\sum_{j=1}^m a_{ij} p_i(t) p_j(t)\right)\,p_i(t),\quad i=1,2,...,m, \quad t>0.
\end{equation}
The dynamical system \eqref{rd1} actually describes the mechanism that individuals tend to switch to strategies that are doing well, or that individuals bear offspring who tend to use the same strategies as their parents, and the fitter the individual, the more numerous his offspring.

Most of the work on replicator dynamics has focused on games that have a
finite strategy space, thus leading to a dynamical system for the frequencies of the population which is finite dimensional.
However,  interesting applications arise either in biology or economics  where the strategy space is not finite or, even, not discrete, see  \cite{B90, MS82, OR01, OR02}. In case the strategy space $S$ is discrete but consisting of an infinite number of strategies, e.g. $S=\Z, $ then the replicator dynamics describing the evolution of the infinite dimensional vector $p(t)=(..., p_1(t), p_2(t),...)$ is described by the following
\[
\frac{dp_i}{dt}=\left(\sum_{j\in \Z}a_{ij} p_j(t)-\sum_{j\in \Z}\sum_{i\in \Z} a_{ij} p_i(t) p_j(t)\right)\,p_i(t),\quad t>0,
\]
which is a infinite dynamical system with $p_i(t)\geq 0$ for $i\in \Z$ and $||p(t)||_{\ell^1(\Z)}=1$ for any $t>0.$

In the current paper we are concentrating on games whose pure strategies belong to a continuum. For instance,  this could be the aspiration level of a player or the size of an investment in economics or it might arise in situations where the pure strategies correspond to geographical points as in economic geography, \cite{PK96}. On the other hand, in biology such strategies correspond to some continuously varying trait such as the sex ratio in a litter or the virulence of an infection, \cite{hs03}. There are different ways of modelling the evolutionary dynamics in this case, however in the current work we adapt the approach introduced in \cite{B90}. In that case the strategy set $\Omega$ is an arbitrary, not necessarily bounded, Borel set of $\R^N, N\geq 2,$ hence strategies can be identified by $x \in \Omega$. For the case of symmetric
two-player games, the pay-off can be given by a Borel measurable function $f: \Omega \times \Omega \rightarrow \R,$ where $f(x,y)$ is the pay-off for
player $1$ when she follows strategy $x$ and player $2$ plays strategy $y.$ A population is now characterized  by its state, a
probability measure $\mathcal{P}$ in the measure space $(\Omega,{\mathcal A})$ where ${\mathcal A}$ is the Borel algebra of subsets of $\Omega.$ The average (mean)
pay-off of a sub-population in state $\mathcal{P}$ against the overall population in state $\mathcal{Q}$ is given by the form
$$E(\mathcal{P},\mathcal{Q}):=\int_{\Omega}\int_{\Omega} f(x,y) \mathcal{Q}(dy) \mathcal{P}(dx).$$
Then, the success (or lack of success) of a strategy $x$ followed by population $\mathcal{Q}$ is provided by the difference
$$\sigma (x,\mathcal{Q}):=\int_{\Omega}f(x,y) \mathcal{Q}(dy)-\int_{\Omega}\int_{\Omega} f(x,y) \mathcal{Q}(dy) \mathcal{Q}(dx)=E(\delta_x,\mathcal{Q})-E(\mathcal{Q},\mathcal{Q}),$$
where $\delta_x$ is the unit mass concentrated on the strategy $x.$

The evolution in time of the population state $\mathcal{Q}(t)$ is given by the replicator dynamics  equation
\begin{equation}\label{rd2}
\frac{d\mathcal{Q}}{dt}(A)=\int_{A} \sigma (x,\mathcal{Q}(t)) \mathcal{Q}(t)(dx),\;t>0,\quad \mathcal{Q}(0)=\mathcal{P},
\end{equation}
for any $A \in \mathcal {A},$ where the time  derivative should be understood with respect to the variational norm of a subspace of the linear span $\mathcal{M}$ of ${\mathcal A}.$  The well-posedeness of \eqref{rd2} as well as relating stability issues were investigated in \cite{OR01,OR02} under the assumption that the pay-off function $f(x,y)$ is bounded.

The abstract form of equation \eqref{rd2} does not actually allow us to obtain insight on the form of its solutions and thus a better understanding of the evolutionary dynamics of the corresponding game. In order to have a better overview of the evolutionary  game, following the approach in \cite{KPY08, KPXY10},  we restrict our attention
to measures $\mathcal{Q}(t)$ which, for each $t>0$, are absolutely continuous with respect to the Lebesgue measure, with probability density $u(x,t).$ Then the replicator dynamics equation \eqref{rd2} can be reduced to the following integro-differential equation
\be{non1}
\frac{\partial u}{\partial t}=\left(\int_{\Omega}f(x,y)u(y,t)\,dy-\int_{\Omega}\int_{\Omega} f(z,y)u(y,t) u(z,t) dy\,dz\right)u(x,t),\;t>0,\;\;x\in \Omega,
\ee
for the density $u.$

There are applications both in biology as well as in computer science where the pay-off kernel has the form $f(x,y)=G(x-y)$ with $G$ being a steep function of Gaussian type, see \cite{HI95, HI98, KO, M11}. This case, in general, models games where the pay-off is measured as the distance from some reference strategy and finally under some proper scaling leads to
\begin{equation}\label{rd3}
\int_{\Omega}f(x,y)u(y,t)\,dy\approx\Delta u(x,t),
\end{equation}
(see also \cite{KP11}) which by virtue of \eqref{rd2}  yields
\begin{equation}\label{rd4}
\frac{\partial u}{\partial t}\approx\left(\Delta u-\int_{\Omega}u\,\Delta u\,dx \right)u.
\end{equation}
Another alternative towards getting pay-offs of this type is to consider
a game with a discrete strategy space and take the appropriate scaling limit. In that case a Taylor expansion and a proper scaling gives a similar approximation to \eqref{rd3}, see also \cite{KPY08, KPXY10}.
%$$
%\int_{\Omega}f(x,y)u(y,t)\,dy\approx \Delta u(x,t)+K(x,t),
%$$
%where $K(x,t)$ contains terms related to first derivatives of $u(x,t)$ and the moments of kernel $\Phi,$ which again leads to \eqref{rd4}, see also \cite{KPY08, KPXY10}.

Therefore in case $\Omega$ is a bounded and smooth domain of $\R^N$ it is easily seen that via integration by parts the nonlocal integro-differential dynamics equation \eqref{non1}  is approximated by the degenerate nonlocal parabolic equation
\be{non2}
\frac{\partial u}{\partial t} = u\Big(\Delta u + \int_\Omega |\nabla u|^2\,dx\Big), \qquad x\in \Omega, \ t>0.
\ee
The nonlocal equation \eqref{non2} is associated with initial condition
\be{inc}
u(x,0)=u_0(x),\;\, x\in \Omega,
\ee
and homogeneous Dirichlet boundary conditions
\be{dbc}
u(x,t)=0,\;\,x\in \partial \Omega, \;\,t>0,
\ee
when the agents avoid to play  the strategies locating on the boundary of the strategy space since they are supposed to be too risky, or the individuals of the biological population do not interact when they are close to the spatial boundary where probably the ``food'' is less.
We remark that when on the boundary of the strategy space individuals do not really distinguish between
nearby strategies and hence populate them equally, then the non-local equation \eqref{non2} should rather
be complemented homogeneous Neumann boundary conditions not explicitly considered here, see \cite{KPXY10}.
%
%
%
%
%
%
%
%
%{\footnotesize
\mysection{Appendix B: A convenient approximation of the initial data}

In the article, we have kept the proof of Lemma \ref{lem:existenceofapproximation} very short. Here we give a more detailed version, which still suppresses some of the more involved technical calculations:
\begin{proof}
Choose $\gamma>0$ and a domain $U_\thetaa\sub \Om$ such that $\dist(U_\thetaa,\dom)>\gamma$. Let $\thetaa\in C_0^\infty(U_\thetaa)$ with $\thetaa\geq 0$ and $\io \thetaa=1$. Let $\eps>0$ and let $\phii\in C_0^\infty(\Om)$ be such that $\norm[W^{1,2}(\Om)]{\phii-u_0}<\eps$ and $\norm[\Phi,\infty]{\phii}\leq C+\zeta(\eps)$, where $\zeta\colon [0,\infty)\to[0,\infty)$ is a function satisfying $\lim_{\eps\to 0} \zeta(\eps)=0$. In order to see that this is possible, recall how smooth approximations $\phii$ of $W^{1,2}(\Om)$-functions $u_0$ are usually constructed (\cite[I §3]{Wloka}):
With the aid of a partition of unity $\set{\al_i}$, the function is written as sum, where the single summands are supported in small patches only and those close to the boundary are shifted towards the interior by application of shift operators $s_i$; finally the function is smoothened by convolution with a standard mollifier $j_\eps$.

We observe that the same procedure applied to $\Phi$ does not violate the inequality $\norm[\Phi,\infty]{u_0}\leq C$, i.e. $u_0\leq C\Phi$, too much, that is
\[
 C\sum j_\eps\star(\al_i s_i(\Phi)) \leq C\sum \al_i\Phi + \zeta \Phi=C\Phi+\zeta \Phi,
\]
holds for some $\zeta$ with $\lim_{\eps\searrow 0} \zeta(\eps)=0$. (The calculations showing this use the fact that mollification of smooth functions converge in $C^1$, that $\Phi$ grows towards the interior, and the Mean Value Theorem.)
Hence the fact that mollification preserves pointwise estimates that hold everywhere shows that also $\phii$ satisfies $\phii(x)\leq C\Phi(x)$.

Let $K$ be a compact subset of $\Om$ such that $|\Om\setminus K|<\eps$ and $\dist(\dom,K)<\eps$.
Let $\rho\in C_0^\infty(\Om)$ such that $\rho=1$ on $\hat K \cup \supp\phii$ and $|\na\rho(x)|<\frac2{\dist(\hat K,\dom)}$ and $0\leq \rho\leq 1$.
Denoting
\begin{align*}
 A=A(\eps)=&\io \Phi^2|\na\rho|^2+\io (1-\rho)^2|\na\Phi|^2+\io |\na\thetaa|^2\left(\io (1-\rho)\Phi\right)^2\\
 B=B(\eps)=&-1-2\io(1-\rho)\Phi\io\na\phii\na\thetaa-2\io(1-\rho)\Phi\io(u_0-\phii)\io|\na\thetaa|^2+2\eps|\Om|\io(1-\rho)\Phi\io|\na\thetaa|^2\\
 \Gamma=\Gamma(\eps)=&\io|\na\phii|^2+2\io(u_0-\phii)\io\na\phii\na\thetaa-2\eps|\Om|\io\na\phii\na\thetaa-2\eps|\Om|\io(u_0-\phii)\io|\na\thetaa|^2\\&+\left(\io(u_0-\phii)\right)^2\io|\na\thetaa|^2+\eps^2|\Om|^2\io|\na\thetaa|^2
\end{align*}
we let
\(
 C=C(\eps)=-\frac{2\Gamma}{B-\sqrt{B^2-4A\Gamma}}.
\) Then $C$ solves
\begin{equation}
 AC^2+BC+\Gamma=0\label{eq:Ceq}.
\end{equation}

As $\Phi$ and $\na \Phi$ are bounded, $1-\rho$ is supported on a small set with measure smaller than $\eps$, and $\Phi|\na\rho|\leq 2D_2$, where $\Phi(x)\leq D_2 dist(x,\dom)$, most integrals from the definition of $A, B, \Gamma$ can be estimated, yielding $A\to 0$, $B\to -1$, $\Gamma\to \io |\nabla u_0|^2$ as $\eps\to 0$.
Therefore,
\[
 C= -\frac{2\Gamma}{B-\sqrt{B^2-4A\Gamma}}\to -\frac{2\io|\na u_0|^2}{-1-\sqrt{1-0}} = \io|\na u_0|^2>0,
\]
as $\eps\to 0$, and in particular,
\(
 \limsup (C-L)\leq 0.
\)
Furthermore, for sufficiently small $\eps$, we have $C>0$. We also observe that
\[
\al=\io(u_0-\phii)-\eps|\Om|-C\io(1-\rho)\Phi\to 0,
\]
as $\eps\to 0$.
If $\eps$ is small enough, therefore, $|\al|<\frac{\frac12 essinf_{\set{x; \dist(x,\dom)>\frac \gamma2}} u_0}{\sup\thetaa}$ and hence \\
\( |\al\thetaa|\leq \frac12 \inf_{\set{x\in\Om, \dist(x,\dom)>\frac\gamma2}}\phii\) on $\Om$ (as $\supp\thetaa\subset\supp\phii$).
Therefore, \begin{equation}\label{eq:lowerbd}\phii(x)+\al\thetaa(x)\geq \frac12 \inf_{\set{x\in\Om;\dist(x,\dom)>\frac d2}}\phii=:C_K
           \end{equation}
 for $x\in K$ and \begin{equation}\label{eq:nonnegative} \phii+\al\thetaa\geq 0\end{equation} on $\Om$, because $\phii\geq 0$ and $\al\thetaa\neq 0$ only on $U_\thetaa$, where \eqref{eq:lowerbd} guarantees \eqref{eq:lowerbd} already. We also have
\begin{align*}
 \phii+\al\thetaa\leq& (L+\frac\eps2)\Phi +\al\thetaa \leq (L+\frac\eps2)\Phi +\thetaa \io |u_0-\phii|\\
 %\leq& (L+\frac\eps2)\Phi +\eps \sqrt{|\Om|} \thetaa\leq  (L+\frac\eps2)\Phi +\eps \sqrt{|\Om|} \sup\thetaa \frac1{\inf_{U_\thetaa}\Phi}\Phi\\
 \leq& (L+\zeta(\eps))\Phi
\end{align*}
with some $\zeta$ fulfilling $\lim_{\eps\searrow 0} \zeta(\eps)=0$. Finally, define
\begin{equation}\label{eq:defue}
 u_{0\eps}=\eps+C(1-\rho)\Phi+\rho(\phii+\alpha\thetaa).
\end{equation}
Estimate \eqref{eq:nonnegative}, the positivity of $C$ and of $\Phi$ in $\Om$ together with \eqref{eq:defue} entail $\uen\geq \eps$.
Accordingly \eqref{a1} holds, for we clearly obtain $u_{0\eps}=\eps$, and $\Lap u_{0\eps}=-C=-\io|\na\uen|^2$ on $\dOm$, because
\bgee
 \io |\na u_{0\eps}|^2&=&\io|\na(\eps+C(1-\rho)\Phi+\rho(\phii+\al\thetaa)|^2%\\
% &=&\io|-C\na\rho\Phi+C(1-\rho)\na\Phi+\na\rho\phii+\al\thetaa\na\rho+\rho\na\phii+\rho\al\na\thetaa|^2\\
% %  =&C^2\io|\na\rho|^2\Phi^2+C^2\io(1-\rho)^2|\na\Phi|^2+\io|\na\rho|^2\phii^2+\al^2\io|\na\rho|^2\thetaa^2+\io\rho^2|\na\phii|^2\\
% %  &+\al^2\io\rho^2|\na\thetaa|^2-2C^2\io(1-\rho)\Phi\na\rho\na\Phi-2C\io|\na\rho|^2\Phi\phii-2C\al\io\Phi\thetaa|\na\rho|^2\\
% %  &-2C\io\Phi\na\rho\na\phii -2»\al\io\rho\Phi\na\rho\na\thetaa +2C\io(1-\rho)\phii\na\Phi\na\rho+2C\al\io(1-\rho)\thetaa\na\Phi\na\rho\\
% %  &+2C\io\rho(1-\rho)\na\Phi\na\phii+2C\al\io(1-\rho)\rho\na\Phi\na\thetaa+2\io\rho\na\rho\phii\na\phii+2\al\io\na\rho\thetaa\rho\na\phii\\
% %  &+2\al\io\rho\phii\na\rho\na\thetaa+2\al^2\io\rho\thetaa\na\rho\na\thetaa+2\al\io\rho^2\na\phii\na\thetaa+2\al\io|\na\rho|^2\phii\thetaa \\
% %  &=C^2\io|\na\rho|^2\Phi^2+C^2\io(1-\rho)^2|\na\Phi|^2+\io|\na\phii|^2+\al^2\io|\na\thetaa|^2+2\al\io\na\phii\na\thetaa \\
% %   &=C^2\io|\na\rho|^2\Phi^2+C^2\io(1-\rho)^2|\na\Phi|^2+\io|\na\phii|^2\\
% %   &+\left(\io(u_0-\phii)-\eps|\Om|-C\io(1-\rho)\Phi\right)^2\io|\na\thetaa|^2\\
% %   &+2\left(\io(u_0-\phii)-\eps|\Om|-C\io(1-\rho)\Phi\right)\io\na\phii\na\thetaa\\
%&=&
=AC^2+(B+1)C+\Gamma=C
\egee
by \eqref{eq:Ceq}. Furthermore,
\bgee
 \io u_{0\eps} &=& \io \eps+ \io C(1-\rho)\Phi+ \io \rho\phii + \alpha \io\rho \thetaa=\io u_0
%  =&\eps|\Om| + C\io(1-\rho)\Phi +\io\rho\phii +\alpha \io 1 \thetaa\\
%  =&\eps|\Om| + C\io(1-\rho)\Phi +\io \phii +\alpha \\
%  =&\eps|\Om| + C\io(1-\rho)\Phi +\io\phii + \io(u_0-\phii)-\eps|\Om|-C\io(1-\rho)\Phi\\
% &=&\io u_0,
\egee
that is \eqref{a6}. The smoothness assertion follows from the smoothness of $\phii$ (as mollification) and $\Phi$ and that of $\rho,\thetaa\in C_0^\infty(\Om)$.
By definition of $\uen$,
\[
 \norm[\Phi,\infty]{u_{0\eps}-\eps}=\norm[\Phi,\infty]{C\Phi(1-\rho)+\rho(\phii+\al\thetaa)}.
\]
In every point $x\in\Om$, $\uen-\eps$ is a convex combination of $C\Phi$ and $\phii+\al\thetaa$, which both satisfy the estimate ``$\leq (L+\zeta(\eps))\Phi$''. Therefore \eqref{ae} holds. Furthermore,
\begin{align*}
 \norm[W^{1,2}(\Om)]{u_{0\eps}-u_0}=&\norm[W^{1,2}(\Om)]{\eps+C\Phi(1-\rho)+\rho(\phii+\al\thetaa)-u_0}\\
 \leq&\eps\sqrt{|\Om|}+ C\norm[L^2(\Om)]{\na\Phi(1-\rho)}+C\norm[L^2(\Om)]{\Phi\na\rho}+C\norm[L^2(\Om)]{\Phi(1-\rho)}\\
% &+ \norm[W^{1,2}(\Om)]{\rho\phii-u_0}+\al \norm[L^2(\Om)]{\rho\thetaa}+\al\norm[L^2(\Om)]{\na\rho\thetaa}+\al\norm[L^2(\Om)]{\rho\na\thetaa}\\
%  \leq&\eps\sqrt{|\Om|}+ C\sup|\na \Phi| \sqrt{|\Om\setminus K|}+C\norm[L^2(\Om)]{\Phi\na\rho}+C\sup\Phi \sqrt{|\Om\setminus K|}\\&+
%  \norm[W^{1,2}(\supp\phii)]{\rho\phii-u_0}+\norm[W^{1,2}(\Om\setminus\supp\phii)]{\rho\phii-u_0}+\al \norm[L^2(\Om)]{\thetaa}+0+\al\norm[L^2(\Om)]{\na\thetaa}\\
% \leq&\eps\sqrt{|\Om|}+ C\sup|\na \Phi| \sqrt{\eps}+2CD_2\sqrt{\eps} +C\sup\Phi \sqrt{\eps}\\&+
%  \norm[W^{1,2}(\Om)]{\phii-u_0}+\norm[W^{1,2}(\Om\setminus K)]{u_0}+\al \norm[W^{1,2}(\Om)]{\thetaa}\\
\leq&\eps\sqrt{|\Om|}+ C\sup|\na \Phi| \sqrt{\eps}+2CD_2\sqrt{\eps} +C\sup\Phi \sqrt{\eps}
+ \eps+\eps+\al \norm[W^{1,2}(\Om)]{\thetaa} \to 0
 \end{align*}
 as $\eps\searrow 0$,
 where we have, once again, used that
 \(
  \norm[L^2(\Om)]{\Phi\na\rho}\leq 2D_2\sqrt{\eps},
 \)
%longer:
% %  \begin{align*}
% %   \norm[L^2(\Om)]{\Phi\na\rho}&=\left(\io \Phi^2|\na\rho|^2 \right)^{\frac12}\leq \left(\int_{\Om\setminus\supp\phii} D_2^2\dist(x,\dom) \frac{4}{\dist(\supp\phii,\dom)^2} dx\right)^{\frac12}\\
% %   &\leq 2 D_2 \sqrt{|\Om\setminus\supp\phii|}\leq  2 D_2 \sqrt{|\Om\setminus K|}\leq 2D_2\sqrt{\eps},
% %  \end{align*}
 as well as $\norm[W^{1,2}(\Om\setminus K)]{u_0}<\eps$ and $\norm[W^{1,2}(\Om)]{u_0-\phii}<\eps$.
 In total, we obtain \eqref{a5}.
Finally, given $K\subset\subset \Om$, the estimate in \eqref{a3} holds for $0<\eps<\dist(K,\dom)$ and with the choice of $C_K$ as in \eqref{eq:lowerbd}.
\end{proof}
%}
%
%
%
%
%
%
%
%

\vspace*{5mm}

{\bf Acknowlegement.}  \quad N.I.~Kavallaris would like to thank Professors V.G.~Papanicolaou and A.N.~Yannacopoulos for introducing him to the topic of infinite dimensional replicator dynamics. N. I. Kavallaris is grateful to  Institut f\"ur Mathematik, Universit\"at Paderborn for its hospitality and stimulating atmosphere during the preparation of part of this research paper.

\end{document}